\documentclass[12pt]{amsart}
\usepackage{amscd,amsmath,amssymb,enumerate,amsfonts,mathrsfs,txfonts}
\usepackage{comment}
\usepackage{hyperref}
\hypersetup{
    colorlinks=true,      
    urlcolor=blue}
\usepackage{xcolor}
\usepackage{tikz,pgfplots,pgf}
\usepackage{mathtools}          
\usepackage{tikz-cd}            
\usepackage[all]{xy}
\usepackage{lipsum}

\tikzset{node distance=3cm, auto}
\usetikzlibrary{calc} 
\usetikzlibrary{trees} 

\usepackage{vmargin}
\setpapersize{A4}
\setmargins{2.5cm}      
{1.5cm}                        
{16.5cm}                      
{23.42cm}                   
{10pt}                          
{1cm}                           
{0pt}                          
{2cm}

\newtheorem{theorem}{Theorem}[section]
\newtheorem{proposition}[theorem]{Proposition}
\newtheorem{definition}[theorem]{Definition}
\newtheorem{corollary}[theorem]{Corollary}

\newtheorem{lemma}[theorem]{Lemma}
\newtheorem{remark}[theorem]{Remark}

\newtheorem{example}[theorem]{Examples}

\def\A{\mathcal{A}}

\def\F{\mathcal{F}}

\def\I{\mathcal{I}}

\def\C{\mathcal{C}}
\def\L{\mathcal{L}}

\def\R{\mathbb{R}}

\def\U{\mathcal{U}}

\def\T{\mathcal{T}}

\def\C{\mathcal{C}}
\def\K{\mathbb{K}}
\def\N{\mathbb{N}}

\def\Lip{\mathrm{Lip}_0}
\def\Lipnorm{\mathrm{Lip}}

\def\Rad{\mathrm{Rad}}
\def\RAD{\mathrm{RAD}}

\allowdisplaybreaks

\begin{document}

\title[Lipschitz stable sequence classes]{Lipschitz stable sequence classes: an approach to Rademacher type and cotype of Lipschitz functions}


\author[D. Ruiz-Casternado]{D. Ruiz-Casternado}
\address[David Ruiz-Casternado]{Departamento de Matem{\'a}ticas, Universidad de Almer{\'i}a, Ctra. de Sacramento s/n, 04120 La Ca\~{n}ada de San Urbano, Almer\'ia, Spain.}
\email{drc446@ual.es}

\begin{abstract}
    In this paper, we extend sequence-class methods from linear and multilinear theory to the Lipschitz setting, highlighting the substantial differences that arise from the lack of linearity. First, we establish a general criterion for lifting a Lipschitz mapping between Banach spaces to a Lipschitz operator between sequence spaces, and we use it to define the class $\Pi_{(Z,Y)}^{\Lip}$ of $(Z,Y)$-summing Lipschitz functions, where $Z$ and $Y$ are sequence classes. We then introduce the notion of Lipschitz stable sequence class and show that $[\Pi_{(Z,Y)}^{\Lip},\pi_{(Z,Y)}^{\Lipnorm}]$ is a Banach Lipschitz ideal whenever the sequence classes satisfy this property. As applications, we present Rademacher type and cotype for Lipschitz functions and identify them with $(Z,Y)$-summing Lipschitz spaces for concrete choices of sequence classes. We prove that the type case forms a Banach Lipschitz ideal, whereas the cotype case does not, and we analyse composition and maximality for the type ideals.

		\vspace{5pt}

		\vspace{5pt}
	\end{abstract}
	
\maketitle

\section*{Introduction}

The theory of transformation of vector-valued sequences by linear and multilinear operators has been carried out especially in the last decade (see, e.g., \cite{BotCam-17, BotSan-25, BotSan-252, BotWoo-23, GalVil-15}). This thread intertwines Banach space geometry, probability in Banach spaces and the theory of operator ideals. The origins of these topics can be traced back to the mid-twentieth century, when the interplay between linear operators and summability properties of sequences began to be systematically explored, eventually leading to a deep and highly structured theory whose multilinear extensions are still under active development.

To be precise, the study of vector-valued sequence transformations by linear operators dates back to the 1950s thanks to Grothendieck \cite{Gro-55}, particularly through his inequality and the concept of nuclear operator. However, it was during the 1960s and 1970s that the theory gained a definitive form, especially with the contributions of Pietsch concerning absolutely $p$-summing operators. In this context, a linear operator is absolutely $p$-summing if it maps weakly $p$-summable sequences into absolutely $p$-summable sequences. This perspective—defining operator classes through the transformation of vector-valued sequences—proved to be remarkably fruitful, since it allowed the introduction of a wide variety of operator ideals, each corresponding to different sequence classes, such as weakly summable, unconditionally summable, or almost unconditionally summable sequences (see, e.g., \cite{BotBraJun-01, BotPelRue-14, DieJarTon-95, JunMat-98, Pie-80}).

Parallel to these developments, the notions of Rademacher type and cotype emerged in the early 1970s from the study of vector-valued random variables. The works of Dubinsky, Pe{\l}czy{\'n}ski and Rosenthal \cite{DubPelRos-72}, Maurey \cite{Mau-74}, and Hoffmann-J{\o}rgensen and Pisier \cite{HofPis-76} established that the behavior of sums of independent random variables in Banach spaces is deeply connected with intrinsic geometric properties of such spaces. These concepts quantify how closely the space resembles a Hilbert space from the probabilistic viewpoint. In fact, soon after their appearance, Rademacher type and cotype were extended from Banach spaces to linear operators, and recently to the multilinear setting. These classes have been extensively studied, and their properties are deeply connected with other linear and multilinear operator ideals (see, e.g., \cite{DieJarPie-01, Pis-86} for the linear case and \cite{BotCam-16, BotCamNas-24, BotFre-24, CamNasSou-25} for the multilinear setting). Specifically, thanks to the theory developed in \cite{BotCam-17}, we know that these subclasses of linear/multilinear operators are closely linked to the transformation of vector-valued sequences.

Our goal in this manuscript is to extend such study to the Lipschitz setting, as a natural step within the theory of nonlinear functional analysis. However, this passage from linear and multilinear operators to Lipschitz mappings is not merely formal: the lack of linear structure produces substantial differences in the sequence-class approach. In particular, the criterion that ensures a Lipschitz map between Banach spaces induces an operator between sequence spaces turns out to be weaker than its multilinear counterpart, and this phenomenon leads naturally to the notion of Lipschitz stable sequence class. Therefore, this paper should serve as a starting point for the systematic study of the transformation of vector-valued sequences by Lipschitz mappings.

Let $X$ and $E$ be Banach spaces. In this paper, our subject area is the set of zero-preserving Lipschitz functions from $X$ into $E$, denoted by $\Lip(X,E)$. It is well-known that it becomes a Banach space if we endow it with the Lipschitz norm
$$
\Lipnorm(f) = \sup_{\substack{x,y \in X\\ x\neq y}} \frac{\|f(x)-f(y)\|}{\|x-y\|} \qquad (f\in\Lip(X,E)).
$$

For the special case when $E$ reduces to the scalar field $\K$, we use the shorthand $\Lip(X) = \Lip(X,\K)$.

Building upon this framework, the concept of Lipschitz ideal arises as a nonlinear extension of the classical theory of operator ideals. The systematic study of such ideals was initiated by Farmer and Johnson in \cite{FarJoh-09}, who introduced the class of Lipschitz $p$-summing operators, extending the classical absolutely $p$-summing operators to the nonlinear context. This theory was subsequently formalised and developed in \cite{AchDahTur-20, AchRueSanYah-16, JimSepVil-14, Saa-17}.

After this introduction, Section~\ref{Sec 1} establishes the notation and background that will be used throughout the manuscript, including the main Banach-valued sequence spaces and the linear concepts of $p$-factorable, $p$-integral, and absolutely $p$-summing operators, together with their Lipschitz counterparts. Section~\ref{Sec 2} develops the sequence-class framework in the Lipschitz setting, emphasizing the differences that arise with respect to the linear and multilinear settings due to the absence of linearity. Given $f \in \Lip(X,E)$, we first prove a general criterion describing when the coordinatewise map $(x_n)_{n\in\N}\mapsto (f(x_n))_{n\in\N}$ is well-defined and Lipschitz. Based on this criterion, we define the class $\Pi_{(Z,Y)}^{\Lip}$ of $(Z,Y)$-summing Lipschitz mappings, where $Z$ and $Y$ are sequence classes. In particular, a sequence class $Z$ will be called Lipschitz stable if the normed space $(\Pi_{(Z,Z)}^{\Lip}(X,E),\pi_{(Z,Z)}^{\Lipnorm})$ coincides with $(\Lip(X,E),\Lipnorm)$. Moreover, some concrete examples are provided and we use these to derive that $[\Pi_{(Z,Y)}^{\Lip},\pi_{(Z,Y)}^{\Lipnorm}]$ is a Banach Lipschitz ideal under suitable assumptions on $Z$ and $Y$. Finally, Section~\ref{Sec 3} studies Rademacher type and cotype of Lipschitz functions in depth: we introduce these notions and characterize them as specific $(Z,Y)$-summing Lipschitz spaces, prove that the type-class forms a Banach Lipschitz ideal whereas the cotype-class does not, and study composition and maximality results for the type ideals.

\section{Background and notation}\label{Sec 1}

 Throughout this paper $X,E$ and $F$ will denote Banach spaces. As usual, $(\L(E,F),\left\|\cdot\right\|)$ stands for the Banach space of all bounded linear operators from $E$ into $F$ endowed with the canonical operator norm. For the special case $F=\K$, we denote by $\L(E,\K)=E^*$ the topological dual space of $E$. Similarly, $(\Lip(X,E),\Lipnorm)$ is the space of all zero-preserving Lipschitz functions from $X$ into $E$ equipped with the Lipschitz norm. Moreover, $\mathrm{FIN}(E)$ and $\mathrm{COFIN}(E)$ correspond to the sets of finite-dimensional subspaces and finite-codimensional closed subspaces of $E$, respectively. With this notation, given $Z \in \mathrm{FIN}(E)$, we denote by $\iota_Z^E$ the canonical inclusion operator from $Z$ into $E$. Analogously, given $Y \in \mathrm{COFIN}(E)$, $q_Y^E$ stands for the canonical quotient operator from $E$ into $E/Y$.

The following sequence spaces will be very useful to us along this work:
\begin{itemize}
	\item $\ell_p(E) = \left\{ (x_n)_{n\in\N} \in E^\N: \|(x_n)_{n\in\N}\|_{\ell_p(E)} = \left(\sum_{i=1}^\infty \|x_i\|_E^p \right)^{\frac{1}{p}} < \infty \right\}$, where $p \in [1,\infty)$.
	\item $\ell_\infty(E) = \left\{ (x_n)_{n\in\N} \in E^\N: \|(x_n)_{n\in\N}\|_\infty = \sup_{n\in\N}\|x_n\|_E  < \infty \right\}$.
	\item $c_0(E) = \left\{ (x_n)_{n\in\N} \in E^\N: \lim_{n\to \infty} \|x_n\|_E =  0 \right\}$.
	\item $\ell_p^w(E) = \left\{ (x_n)_{n\in\N} \in E^\N: \|(x_n)_{n\in\N}\|_{\ell_p^w(E)} = \sup_{\phi \in B_{E^*}} \left(\sum_{i=1}^\infty |\phi(x_i)|^p \right)^{\frac{1}{p}} < \infty \right\}$.
	\item $\Rad(E) = \left\{ (x_n)_{n\in\N} \in E^\N: \forall p\in (0,\infty), \sum_n r_n(t)x_n \text{ converges in } L_p([0,1],E) \right\}$, where $(r_n)_{n\in\N}$ are the Rademacher maps. Even more, $\Rad(E)$ becomes a Banach space endowed with the norm
	$$
	\|(x_n)_{n\in\N}\|_{\Rad(E)} = \left( \int_{0}^{1}\left\| \sum_{i=1}^{\infty} r_i(t)x_i \right\|_E^2 dt \right)^{\frac{1}{2}} \qquad ((x_n)_{n\in\N}\in \Rad(E)).
	$$
	\item $\RAD(E) = \left\{ (x_n)_{n\in\N} \in E^\N: \|(x_n)_{n\in\N}\|_{\RAD(E)} = \sup_{N\in\N}\|(x_k)_{k=1}^N\|_{\Rad(E)} < \infty \right\}$.
\end{itemize}

For additional details on these sequence spaces we refer the reader to \cite{BlaBotPelRue-11, BlaFouSch-07, DieJarTon-95}. Furthermore, $c_{00}(E)$ stands for the set of all finite sequences of elements of $E$ and $(e_n^{(j)})_{n\in\N}$ is the sequence with a $1$ in the $j$-th coordinate and zeros elsewhere.

Given $p\in[1,\infty)$, a bounded linear operator $T \in \L(E,F)$ is said to be:
\begin{itemize}
	\item $p$-factorable (see, e.g., \cite[19.4]{Pie-80}) if there is a factorization $\kappa_F \circ T: E \xrightarrow{S} L_p(\Omega,\mu) \xrightarrow{R} F^{**}$, where $\kappa_F$ is the canonical linear embedding of $F$ into its bidual, $(\Omega,\mu)$ is a measure space, $S\in\L(E,L_p(\Omega,\mu))$ and $R\in\L(L_p(\Omega,\mu),F^{**})$. We denote by $\F_p(E,F)$ the space of all $p$-factorable linear operators from $E$ into $F$, and it is a Banach space endowed with the norm $\|T\|_{\F_p} = \inf\{\|S\|\|R\|\}$ where the infimum is taken over all such factorizations of $T$. In \cite{AchTia-24} this Banach operator ideal was adapted to the Lipschitz setting, and denoted $[\F_p^{\Lip},\left\|\cdot\right\|_{\F_p^{\Lipnorm}}]$.
	\item $p$-integral (see, e.g., \cite[19.2]{Pie-80}) if there exists a factorization $\kappa_F\circ T: E \xrightarrow{S} L_\infty(\Omega,\mu)\xrightarrow{I_p} L_p (\Omega,\mu) \xrightarrow{R} F^{**}$, where $S\in\L(E,L_\infty(\Omega,\mu))$, $R\in\L(L_p(\Omega,\mu),F^{**})$ and $I_p: L_\infty(\Omega,\mu) \to L_p(\Omega,\mu)$ denotes the canonical embedding inclusion. The set of all $p$-integral linear operators from $E$ into $F$ is denoted by $\I_p(E,F)$, and it becomes a Banach space just defining $\|T\|_{\I_p} =\inf\{\|S\|\|R\|\}$ with the infimum being taken over all such factorizations of $T$. In \cite[Section 2.3]{AchRueSanYah-16} this notion was extended to the Lipschitz framework, and we denote by $[\I_p^{\Lip},\left\|\cdot\right\|_{\I_p^{\Lip}}]$ the Banach Lipschitz ideal of $p$-integral Lipschitz maps.
	\item absolutely $p$-summing (see, e.g., \cite[17.3]{Pie-80}) if there exists a constant $C>0$ such that
	\begin{equation}\label{eq 5}
	\left(\sum_{k=1}^{N} \|T(x_k)\|^p \right)^{\frac{1}{p}} \leq C \sup_{\phi\in B_{E^*}} \left(\sum_{k=1}^{N} |\phi(x_k)|^p \right)^{\frac{1}{p}}
	\end{equation}
	for all $N\in\N$ and any choice $(x_1,\ldots,x_N)\in E^N$. The set of all absolutely $p$-summing linear operators from $E$ into $F$ is denoted by $\Pi_p(E,F)$, and it becomes a Banach space endowed with the norm $\pi_p(T) = \inf\{C: (\ref{eq 5}) \text{ is satisfied} \}$. This notion was extended to the Lipschitz setting in \cite{FarJoh-09} and we denote by $[\Pi_p^{\Lip},\pi_p^{\Lipnorm}]$ the Banach Lipschitz ideal of $p$-summing Lipschitz operators.
\end{itemize}


\section{Lipschitz stability of sequence classes}\label{Sec 2}

Let us recall by \cite[Definition 2.1]{BotCam-17} that a sequence class $Z$ is an assignment that associates to each Banach space $E$ a vector subspace of $E^\N$, denoted $Z(E)$, endowed with a complete norm $\left\|\cdot\right\|_{Z(E)}$ such that $c_{00}(E) \subseteq Z(E) \subseteq \ell_\infty(E)$ with $\left\|\cdot\right\|_\infty \leq \left\|\cdot\right\|_{Z(E)}$, and $\|(e_n^{(j)})_{n\in\N}\|_{Z(\K)}=1$ for all $j \in \N$. Particular examples of sequence classes can be found in \cite[Example 2.2]{BotCam-17}.

If for any $(x_n)_{n\in\N} \in E^\N$ we have that $(x_n)_{n\in\N}\in Z(E)$ if and only if $\sup_{N\in\N} \|(x_k)_{k=1}^N\|_{Z(E)}<\infty$, then the sequence class $Z$ is called finitely determined. In such a case $\|(x_n)_{n\in\N}\|_{Z(E)}=\sup_{N\in\N}\|(x_k)_{k=1}^N\|_{Z(E)}$. $\RAD,\ell_p$ and $\ell_p^w$ for $p\in [1,\infty]$ are examples of finitely determined sequence classes.

Firstly, we analyse the transformation of Banach-valued sequences by means of Lipschitz functions.

\begin{theorem}\label{Th 1}
	Let $f\in \Lip(X,E)$ and let $Z,Y$ be sequence classes. Consider the following assertions:
	\begin{itemize}
		\item[$(i)$] If $(x_n)_{n\in\N},(y_n)_{n\in\N}\in Z(X)$ then $(f(x_n)-f(y_n))_{n\in\N}\in Y(E)$.
		\item[$(ii)$] The operator $\hat{f}: Z(X)\to Y(E)$ given by
		$$
		\hat{f}((x_n)_{n\in\N}) = (f(x_n))_{n\in\N} \qquad ((x_n)_{n\in\N} \in Z(X))
		$$
		is well-defined.
		\item[$(iii)$] $\hat{f} \in \Lip(Z(X),Y(E))$.
		\item[$(iv)$] There exists $C>0$ so that for all $N\in \N$ and all $(x_1,\ldots,x_N)$, $(y_1,\ldots,y_N)\in X^N$, we have
		$$
		\left\|(f(x_k)-f(y_k))_{k=1}^N\right\|_{Y(E)} \leq C\left\|(x_k-y_k)_{k=1}^N\right\|_{Z(X)}.
		$$
	\end{itemize}
	Then $(i) \Leftrightarrow (ii) \Leftarrow (iii) \Rightarrow (iv)$.  Moreover, if the sequence classes $Z$ and $Y$ are finitely determined, then $(iv) \Rightarrow (iii)$. Moreover, in this case, $\Lipnorm(\hat{f}) = \inf\{C\}$ with the infimum being taken over all such constants as in $(iv)$.
\end{theorem}

\begin{proof}
	$(i) \Rightarrow (ii):$ We begin by showing that $\hat{f}$ is a function. Towards this end, let $(x_n)_{n\in\N},(y_n)_{n\in\N} \in Z(X)$ be such that $\hat{f}((x_n)_{n\in\N}) \neq \hat{f}((y_n)_{n\in\N})$. Then $(f(x_n))_{n\in\N} \neq (f(y_n))_{n\in\N}$, i.e., there is $k\in\N$ for which $f(x_k)\neq f(y_k)$, and since $f$ is a function, we have $x_k \neq y_k$. Hence $(x_n)_{n\in\N} \neq (y_n)_{n\in\N}$. Moreover, we claim that $\hat{f}$ is well-defined. Indeed, let $(x_n)_{n\in\N} \in Z(X)$. Since $Z(X)$ is a vector space, the zero sequence $(0)_{n\in\N} \in Z(X)$. Thus applying our hypothesis we have
	$$
	\hat{f}((x_n)_{n\in\N})=(f(x_n))_{n\in\N}=(f(x_n)-f(0))_{n\in\N} \in Y(E).
	$$
	Consequently, $\hat{f}$ is a well-defined operator.

	$(ii) \Rightarrow (i):$ Assume that $\hat{f}:Z(X) \to Y(E)$ is a well-defined operator. Let $(x_n)_{n\in\N},(y_n)_{n\in\N} \in Z(X)$ and note that
	$$
	(f(x_n)-f(y_n))_{n\in\N} = \hat{f}((x_n)_{n\in\N})-\hat{f}((y_n)_{n\in\N}) \in Y(E).
	$$

	$(iii) \Rightarrow (i):$ Let us suppose that $\hat{f} \in \Lip(Z(X),Y(E))$, and let $(x_n)_{n\in\N},(y_n)_{n\in\N} \in Z(X)$. Then
	$$
	\|(f(x_n)-f(y_n))_{n\in\N}\|_{Y(E)}=\|\hat{f}((x_n)_{n\in\N})-\hat{f}((y_n)_{n\in\N})\|_{Y(E)} \leq \Lipnorm(\hat{f}) \|(x_n)_{n\in\N}-(y_n)_{n\in\N}\|_{Z(X)} < \infty,
	$$
	and we obtain that $(f(x_n)-f(y_n))_{n\in\N} \in Y(E)$.

	$(iii) \Rightarrow (iv):$ Assume that $\hat{f}\in \Lip(Z(X),Y(E))$. Let $N\in\N$ and $(x_1,\ldots,x_N)$, $(y_1,\ldots,y_N)\in X^N$. Let us define the sequences $(u_n)_{n\in\N}$ and $(v_n)_{n\in\N}$ by
	$$
	u_k = \begin{cases}
		x_k \quad &k \leq N,\\
		0 \quad &k > N.
	\end{cases} \qquad
	v_k = \begin{cases}
		y_k \quad &k \leq N,\\
		0 \quad &k > N.
	\end{cases}
	$$

	Since $c_{00}(X) \subseteq Z(X)$ then $(u_n)_{n\in\N},(v_n)_{n\in\N} \in Z(X)$. Therefore,
	$$
	\hat{f}((u_n)_{n\in\N})-\hat{f}((v_n)_{n\in\N}) = (f(u_n)-f(v_n))_{n\in\N} = (f(x_1)-f(y_1), \ldots, f(x_N)-f(y_N), 0, \ldots),
	$$
	and we deduce that
	\begin{align*}
		\|(f(x_k)-f(y_k))_{k=1}^N\|_{Y(E)} &= \|\hat{f}((u_n)_{n\in\N})-\hat{f}((v_n)_{n\in\N})\|_{Y(E)}\\ 
		&\leq \Lipnorm(\hat{f})\|(u_n-v_n)_{n\in\N}\|_{Z(X)}\\ 
		&= \Lipnorm(\hat{f})\|(x_k-y_k)_{k=1}^N\|_{Z(X)}.
	\end{align*}

	Hence we obtain the desired result and $\inf\{C\} \leq \Lipnorm(\hat{f})$, where the infimum is taken over all such positive constants for which the previous inequality is fulfilled.

	$(iv) \Rightarrow (iii):$ Suppose that there exists a constant $C>0$ so that
	$$
	\left\|(f(x_k)-f(y_k))_{k=1}^N\right\|_{Y(E)} \leq C\left\|(x_k-y_k)_{k=1}^N\right\|_{Z(X)}
	$$
	for all $N\in\N$ and all $(x_1,\ldots,x_N)$, $(y_1,\ldots,y_N)\in X^N$. Since $Z$ and $Y$ are finitely determined, just taking the supremum over all such $N\in\N$ we obtain
	$$
	\|\hat{f}((x_n)_{n\in\N})-\hat{f}((y_n)_{n\in\N})\|_{Y(E)} \leq C\|(x_n)_{n\in\N}-(y_n)_{n\in\N}\|_{Z(X)}.
	$$
	Hence $\hat{f} \in \Lip(Z(X),Y(E))$ and $\Lipnorm(\hat{f}) \leq \inf\{C\}$ with the infimum being taken over all such positive constants satisfying that inequality.
\end{proof}

\begin{remark}
	If one compares the above result with its counterpart in the multilinear setting, it could be observed that it is not merely a replica since, by the lack of linear structure, the theorem now obtained is weaker. Specifically, we are going to demonstrate that statement $(i) \Rightarrow (iii)$ in Theorem \ref{Th 1} is not true in general, unlike what happens in the multilinear setting (\cite[Proposition 2.4]{BotCam-17}). Let us consider $X=E=\R$, $Z=\ell_2$ and $Y=\ell_1$, and define the mapping $f:\R \to \R$ by
	$$
	f(x) = \begin{cases}
		x^2 \qquad &|x|\leq 1,\\
		2|x|-1 \qquad &|x| > 1.
	\end{cases}
	$$

	It is not difficult to show that $f \in \Lip(\R)$. We claim that condition $(i)$ in Theorem \ref{Th 1} is satisfied. Indeed, let $(x_n)_{n\in\N},(y_n)_{n\in\N} \in \ell_2(\R)$ and define the set $\Gamma = \{n \in\N: \max\{|x_n|,|y_n|\}>1\}$. Since $(x_n)_{n\in\N},(y_n)_{n\in\N} \in \ell_2(\R)$, we have that $\Gamma$ is finite. Thus
	$$
	\sum_{i=1}^\infty |f(x_i)-f(y_i)| = \sum_{i\in \Gamma} |f(x_i)-f(y_i)| + \sum_{i\notin \Gamma} |f(x_i)-f(y_i)|.
	$$

	Clearly, $\sum_{i\in \Gamma} |f(x_i)-f(y_i)| <\infty$ because $f\in\Lip(\R)$ and the index set is finite. On the other hand, if $i \notin \Gamma$, then $|x_n|,|y_n| \leq 1$, and therefore, by Cauchy--Schwarz inequality,
	\begin{align*}
	\sum_{i\notin \Gamma} |f(x_i)-f(y_i)| &= \sum_{i\notin\Gamma} |x_i^2-y_i^2| = \sum_{i\notin\Gamma}|x_i-y_i||x_i+y_i|\\
	&\leq \left( \sum_{i=1}^{\infty} |x_i-y_i|^2 \right)^{\frac{1}{2}} \left( \sum_{i=1}^{\infty} |x_i+y_i|^2 \right)^{\frac{1}{2}}< \infty.
	\end{align*}

	Thus $(f(x_n)-f(y_n))_{n\in\N} \in \ell_1(\R)$ and condition $(i)$ holds. Nevertheless, for each $N \in \N$, let $(u_n^N)_{n\in\N}$ be the sequence given by
	$$
	u_k^N= 
	\begin{cases}
	1 \quad  & k \leq N,\\
	0 \quad & k > N. 
	\end{cases}
	$$

	It is straightforward that $(u_n^N)_{n\in\N} \in c_{00}(\R) \subseteq \ell_2(\R)$ and then 
	$$
	\|\hat{f}((u_n^N)_{n\in\N})-\hat{f}((0)_{n\in\N})\|_{\ell_1(\R)} = \|(f(u_n^N)-f(0))_{n\in\N}\|_{\ell_1(\R)} = \sum_{k=1}^{N} |f(1)| = N,
	$$

	whereas $\|(u_n^N)_{n\in\N}\|_{\ell_2(\R)} = \left( \sum_{k=1}^{N} |u_k^N|^2 \right)^{\frac{1}{2}} = \sqrt{N}$. Consequently,
	$$
	\frac{\|\hat{f}((u_n^N)_{n\in\N})-\hat{f}((0)_{n\in\N})\|_{\ell_1(\R)}}{\|(u_n^N)_{n\in\N}\|_{\ell_2(\R)}} = \sqrt{N} \xrightarrow{N \to \infty} \infty.
	$$

	Hence $\hat{f} \notin \Lip(\ell_2(\R),\ell_1(\R))$.
\end{remark}

Let us recall by \cite[Definition 2.5]{BotCam-17} that given two sequence classes $Z$ and $Y$, then $Z \prec Y$ if the following conditions are fulfilled:
\begin{itemize}
	\item[$(i)$] $Z(X)$ is a closed subspace of $Y(X)$ which carries the induced norm from $Y(X)$ for all Banach space $X$.
	\item[$(ii)$] Given $(x_n)_{n\in\N} \in Y(X)$, then $(x_n)_{n\in\N} \in Z(X)$ if and only if $\lim_{n,N}\|(x_k)_{k=n}^N\|_{Y(X)}=0$.
\end{itemize}

Just taking into account Theorem \ref{Th 1}, the following result for sequence classes satisfying the previous relation can be shown. Compare it to \cite[Corollary 2.6]{BotCam-17} for the multilinear version, and note the differences with the Lipschitz setting.

\begin{corollary}\label{Cor 2}
	Let $Z_1,Z_2,Y_1,Y_2$ be sequence classes such that $Z_2,Y_2$ are finitely determined with $Z_1 \prec Z_2$ and $Y_1 \prec Y_2$. The following conditions are equivalent for $f\in \Lip(X,E)$:
	\begin{itemize}
		\item[$(i)$] The induced operator $\hat{f} \in \Lip(Z_1(X),Y_1(E))$.
		\item[$(ii)$] The induced operator $\hat{f} \in \Lip(Z_2(X),Y_2(E))$.
	\end{itemize}
	In such a case, $\Lipnorm(\hat{f}:Z_1(X)\to Y_1(E)) = \Lipnorm(\hat{f}:Z_2(X)\to Y_2(E))$.
\end{corollary}

\begin{proof}
	$(i) \Rightarrow (ii):$ First of all, note that $Z_1(X)$ and $Y_1(E)$ are closed subspaces of $Z_2(X)$ and $Y_2(E)$, respectively, carrying the induced norms. Thus by $(iii) \Rightarrow (iv)$ in Theorem \ref{Th 1} there exists a constant $C>0$ such that, given $N\in\N$ and $(x_1,\ldots,x_N)$, $(y_1,\ldots,y_N) \in X^N$, we have
	$$
	\|(f(x_k)-f(y_k))_{k=1}^N\|_{Y_1(E)} \leq C\|(x_k-y_k)_{k=1}^N\|_{Z_1(X)}.
	$$
	Since $c_{00}(X) \subseteq Z_1(X)$ and $c_{00}(E) \subseteq Y_1(E)$, and these spaces carry the induced norms from $Z_2(X)$ and $Y_2(E)$, respectively, the previous inequality can be rewritten as
	$$
	\|(f(x_k)-f(y_k))_{k=1}^N\|_{Y_2(E)} \leq C\|(x_k-y_k)_{k=1}^N\|_{Z_2(X)}.
	$$
	Since $Z_2$ and $Y_2$ are finitely determined sequence classes we conclude this assertion by $(iv) \Rightarrow (iii)$ in Theorem \ref{Th 1}. Moreover, it follows straightforwardly that $\Lipnorm(\hat{f}:Z_2(X)\to Y_2(E)) \leq \Lipnorm(\hat{f}:Z_1(X)\to Y_1(E))$.

	$(ii) \Rightarrow (i):$ Let us assume that $\hat{f} \in \Lip(Z_2(X),Y_2(E))$ and consider $(x_n)_{n\in\N},(y_n)_{n\in\N} \in Z_1(X)$. Since $Z_1 \prec Z_2$ we have that $(x_n)_{n\in\N},(y_n)_{n\in\N} \in Z_2(X)$ and then $(f(x_n)-f(y_n))_{n\in\N} \in Y_2(E)$. We claim that $(f(x_n)-f(y_n))_{n\in\N} \in Y_1(E)$. Indeed, since $Y_1 \prec Y_2$, it is enough to prove that $\lim_{n,N}\|(f(x_k)-f(y_k))_{k=n}^N\|_{Y_2(E)}=0$. Let $N,n\in\N$ with $n\leq N$, and let us define the sequences $(u_m^{n,N})_{m\in\N},(v_m^{n,N})_{m\in\N}$ as
	$$
	u_k^{n,N} = \begin{cases}
		x_k \quad &n \leq k \leq N,\\
		0 \quad &\text{otherwise}.
	\end{cases} \qquad
	v_k^{n,N} = \begin{cases}
		y_k \quad &n\leq k \leq N,\\
		0 \quad &\text{otherwise}.
	\end{cases}
	$$

	Since $c_{00}(X) \subseteq Z_1(X) \subseteq Z_2(X)$, we can ensure that $(u_m^{n,N})_{m\in\N},(v_m^{n,N})_{m\in\N} \in Z_2(X)$ and
	$$
	\hat{f}((u_m^{n,N})_{m\in\N})-\hat{f}((v_m^{n,N})_{m\in\N}) = (0, \ldots, 0, f(x_n)-f(y_n), \ldots, f(x_N)-f(y_N), 0, \ldots).
	$$

	Therefore,
	\begin{align*}
		\|(f(x_k)-f(y_k))_{k=n}^N\|_{Y_2(E)} &= \|\hat{f}((u_m^{n,N})_{m\in\N})-\hat{f}((v_m^{n,N})_{m\in\N})\|_{Y_2(E)}\\
		&\leq \Lipnorm(\hat{f})\|(u_m^{n,N}-v_m^{n,N})_{m\in\N}\|_{Z_2(X)}\\ 
		&= \Lipnorm(\hat{f})\|(x_k-y_k)_{k=n}^N\|_{Z_2(X)}.
	\end{align*}

	Just taking into account that $(x_n-y_n)_{n\in\N} \in Z_1(X)$ and $Z_1 \prec Z_2$, we have that $\lim_{n,N}\|(x_k-y_k)_{k=n}^N\|_{Z_2(X)} = 0$. Consequently $\lim_{n,N}\|(f(x_k)-f(y_k))_{k=n}^N\|_{Y_2(E)}=0$, and since $Y_1 \prec Y_2$, we can conclude that $(f(x_n)-f(y_n))_{n\in\N} \in Y_1(E)$. Thus $\hat{f}:Z_1(X)\to Y_1(E)$ is well-defined by assertion $(i) \Rightarrow (ii)$ in Theorem \ref{Th 1} and
	\begin{align*}
		\|\hat{f}((x_n)_{n\in\N})-\hat{f}((y_n)_{n\in\N})\|_{Y_1(E)} &= \|\hat{f}((x_n)_{n\in\N})-\hat{f}((y_n)_{n\in\N})\|_{Y_2(E)}\\
		&\leq \Lipnorm(\hat{f})\|(x_n-y_n)_{n\in\N}\|_{Z_2(X)}\\
		&= \Lipnorm(\hat{f})\|(x_n-y_n)_{n\in\N}\|_{Z_1(X)}.
	\end{align*}

	Hence $\hat{f} \in \Lip(Z_1(X),Y_1(E))$ with $\Lipnorm(\hat{f}:Z_1(X)\to Y_1(E)) \leq \Lipnorm(\hat{f}:Z_2(X)\to Y_2(E))$.
\end{proof}

The next subclass of Lipschitz mappings is introduced as a result of the relations established in Theorem \ref{Th 1}.

\begin{definition}\label{Def 3}
	Let $Z$ and $Y$ be sequence classes. A function $f\in\Lip(X,E)$ is called $(Z,Y)$-summing if the induced operator $\hat{f} \in \Lip(Z(X),Y(E))$. The set of all $E$-valued $(Z,Y)$-summing Lipschitz functions defined in $X$ is denoted by $\Pi_{(Z,Y)}^{\Lip}(X,E)$. Moreover, we define
	$$
	\pi_{(Z,Y)}^{\Lipnorm}(f) := \Lipnorm(\hat{f}) \qquad \left(f \in \Pi_{(Z,Y)}^{\Lip}(X,E)\right).
	$$
\end{definition}

We will be particularly interested in sequence classes satisfying an additional condition related to the space of Lipschitz functions, in terms of the definition we have just proposed.

\begin{definition}\label{Def 4}
	A sequence class $Z$ is called Lipschitz stable if for every pair of Banach spaces $X$ and $E$, $(\Pi_{(Z,Z)}^{\Lip}(X,E),\pi_{(Z,Z)}^{\Lipnorm}) = (\Lip(X,E),\Lipnorm)$.
\end{definition}

The above definition tells us that a sequence class $Z$ is Lipschitz stable if for every $f\in\Lip(X,E)$, its associated operator $\hat{f}$ belongs to $\Lip(Z(X),Z(E))$ with $\Lipnorm(f)=\Lipnorm(\hat{f})$.

The following auxiliary results on Lipschitz stable sequence classes will be very useful throughout the paper.

\begin{lemma}\label{Lem 5}
	Let $Z$ be a Lipschitz stable sequence class and let $X$ be a Banach space. Then $\|(0,\ldots,0,x-y,0,\ldots)\|_{Z(X)} = \|x-y\|_X$ for all $x,y \in X$ regardless of the position in which $x-y$ appears in the sequence.
\end{lemma}

\begin{proof}
	Let $x,y \in X$ and let us consider the mapping $g_{x-y}: \mathbb{K} \to X$ defined by
	$$
	g_{x-y}(\alpha) = \alpha(x-y) \qquad (\alpha \in \mathbb{K}).
	$$

	It is clear that $g_{x-y}(0)=0$ and
	$$
	\frac{\|g_{x-y}(\alpha)-g_{x-y}(\beta)\|_X}{|\alpha-\beta|} = \frac{\|\alpha(x-y)-\beta(x-y)\|_X}{|\alpha-\beta|} = \frac{\|(\alpha-\beta)(x-y)\|_X}{|\alpha-\beta|} = \|x-y\|_X
	$$
	for all $\alpha,\beta \in \mathbb{K}$ with $\alpha\neq \beta$. Thus $g_{x-y}\in\Lip(\mathbb{K},X)$ with $\Lipnorm(g_{x-y})=\|x-y\|_X$. Fix $j\in\N$ and let $(e_n^{(j)})_{n\in\N}$ be the sequence defined as $0$ if $n\neq j$ and $1$ if $n=j$. Then
	\begin{align*}
		\|x-y\|_X &\leq \|(0,\ldots, 0, x-y^{(j)},0,\ldots)\|_{Z(X)}=\|\hat{g}_{x-y}((2e_n^{(j)})_{n\in\N})-\hat{g}_{x-y}((e_n^{(j)})_{n\in\N})\|_{Z(X)}\\
		&\leq \Lipnorm(\hat{g}_{x-y})\|(e_n^{(j)})_{n\in\N}\|_{Z(\mathbb{K})}=\Lipnorm(\hat{g}_{x-y})=\|x-y\|_X,
	\end{align*}
	where the first inequality is since $Z(X) \subseteq \ell_\infty(X)$ with $\left\|\cdot\right\|_\infty \leq \left\|\cdot\right\|_{Z(X)}$, and the second one is due to $Z$ is Lipschitz stable. Consequently, we obtain the desired result.
\end{proof}

\begin{lemma}\label{Lem 6}
	Let $Z,Y$ be Lipschitz stable sequence classes and let $X,E$ be Banach spaces. Then $\Pi_{(Z,Y)}^{\Lip}(X,E)\subseteq \Lip(X,E)$ with $\Lipnorm(f) \leq \pi_{(Z,Y)}^{\Lipnorm}(f)$ for all $f \in \Pi_{(Z,Y)}^{\Lip}(X,E)$.
	
\end{lemma}

\begin{proof}
	Let $f \in \Pi_{(Z,Y)}^{\Lip}(X,E)$ and let $x,y\in X$ with $x\neq y$. Then, by Lemma \ref{Lem 5}:
	\begin{align*}
		\|f(x)-f(y)\|_E &= \|(f(x)-f(y),0,0,\ldots)\|_{Y(E)} = 
		\|\hat{f}((x,0,0,\ldots))-\hat{f}((y,0,0,\ldots))\|_{Y(E)}\\
		&\leq \Lipnorm(\hat{f})\|(x-y,0,0,\ldots)\|_{Z(X)} = \pi_{(Z,Y)}^{\Lipnorm}(f)\|x-y\|_X.
	\end{align*}
	Hence we obtain that $\Lipnorm(f) \leq \pi_{(Z,Y)}^{\Lipnorm}(f)$.
\end{proof}

Next, we proceed to study the Lipschitz stability of some of the sequence classes presented at the very beginning of this work.

\begin{proposition}\label{Prop 7}
	Let $p \in [1,\infty]$. The sequence class $\ell_p$ is Lipschitz stable.
\end{proposition}

\begin{proof}
	We focus on the case $1\leq p <\infty$ since the case $p=\infty$ follows similarly. Let $f\in\Lip(X,E)$ and $(x_n)_{n\in\N},(y_n)_{n\in\N} \in \ell_p(X)$. Then for each $k\in\N$ we have $\|f(x_k)-f(y_k)\|_E \leq \Lipnorm(f)\|x_k-y_k\|_X$. Therefore by Minkowski's inequality
	$$
	\sum_{i=1}^\infty \|f(x_i)-f(y_i)\|_E^p \leq \sum_{i=1}^\infty \Lipnorm(f)^p \|x_i-y_i\|_X^p \leq \Lipnorm(f)^p \left[ \left( \sum_{i=1}^\infty \|x_i\|_X^p \right)^{\frac{1}{p}} + \left( \sum_{i=1}^\infty \|y_i\|_X^p \right)^{\frac{1}{p}} \right]^p < \infty.
	$$

	Hence $(f(x_n)-f(y_n))_{n\in\N} \in \ell_p(E)$. Moreover
	$$
	\|\hat{f}((x_n)_{n\in\N})-\hat{f}((y_n)_{n\in\N})\|_{\ell_p(E)} = \|(f(x_n)-f(y_n))_{n\in\N}\|_{\ell_p(E)}\leq \Lipnorm(f)\|(x_n-y_n)_{n\in\N}\|_{\ell_p(X)}.
	$$

	Thus $\hat{f}\in \Lip(\ell_p(X),\ell_p(E))$ with $\Lipnorm(\hat{f})\leq \Lipnorm(f)$. For the remaining inequality, let $x,y\in X$ with $x\neq y$ and note that $\|(x-y,0,0,\ldots)\|_{\ell_p(X)} = \|x-y\|_X$. Since $\hat{f}((x,0,0,\ldots))-\hat{f}((y,0,0,\ldots))= (f(x)-f(y),0,0,\ldots)$, we have $\|\hat{f}((x,0,0,\ldots))-\hat{f}((y,0,0,\ldots))\|_{\ell_p(E)}= \|f(x)-f(y)\|_E$. Then
	$$
	\Lipnorm(\hat{f}) \geq \frac{\|\hat{f}((x,0,0,\ldots))-\hat{f}((y,0,0,\ldots))\|_{\ell_p(E)}}{\|(x-y,0,0,\ldots)\|_{\ell_p(X)}} = \frac{\|f(x)-f(y)\|_E}{\|x-y\|_X}.
	$$

	Just taking the supremum over all $x,y \in X$ with $x\neq y$ we conclude that $\Lipnorm(\hat{f}) \geq \Lipnorm(f)$.
\end{proof}

To our surprise, the Lipschitz framework is even more rigid than the multilinear one with respect to the stability of the sequence classes $\Rad$ and $\RAD$ (compare it to \cite[Theorem 4.5]{BotCam-17}).

\begin{proposition}
	The sequence classes $\mathrm{RAD}$ and $\mathrm{Rad}$ are not Lipschitz stable.
\end{proposition}

\begin{proof}
	Let $h: c_0(\R) \to \R$ be defined by
	$$
	h((x_n)_{n\in\N})=\|(x_n)_{n\in\N}\|_\infty \qquad ((x_n)_{n\in\N}\in c_0(\R)).
	$$

	Clearly $h((0)_{n\in\N})=0$, and for every $(x_n)_{n\in\N},(y_n)_{n\in\N} \in c_0(\R)$ we have
	$$
	|h((x_n)_{n\in\N})-h((y_n)_{n\in\N})| = \left|\|(x_n)_{n\in\N}\|_\infty - \|(y_n)_{n\in\N}\|_\infty\right| \leq \|(x_n)_{n\in\N}-(y_n)_{n\in\N}\|_\infty.
	$$

	Thus $h \in\Lip(c_0(\R))$ with $\Lipnorm(h) \leq 1$. We claim that $\RAD$ is not Lipschitz stable. Indeed, let $(e_n)_{n\in\N}$ be the canonical basis of $c_0(\R)$. For every $N\in\N$ and $t \in [0,1]$ we have that $\left\| \sum_{k=1}^N r_k(t)(e_n^{(k)})_{n\in\N} \right\|_\infty =1$. Therefore
	$$
	\|(e_n)_{n\in\N}\|_{\RAD(c_0(\R))} = \sup_{N\in\N}\left( \int_0^1 \left\| \sum_{k=1}^N r_k(t)(e_n^{(k)})_{n\in\N} \right\|_\infty^2 dt \right)^{\frac{1}{2}}=1,
	$$
	so $(e_n)_{n\in\N} \in \RAD(c_0(\R))$. On the other hand, note that
	$$
	\hat{h}((e_n)_{n\in\N}) = (h((e_n^{(1)})_{n\in\N}),h((e_n^{(2)})_{n\in\N}),\ldots) = (1,1,\ldots).
	$$

	Now, for every scalar sequence $(a_n)_{n\in\N}$ we have
	$$
	\|(a_n)_{n\in\N}\|_{\RAD(\R)} = \sup_{N\in\N}\left( \int_0^1 \left| \sum_{k=1}^N a_k r_k(t) \right|^2 dt \right)^{\frac{1}{2}} = \sup_{N\in\N}\left( \sum_{k=1}^N |a_k|^2 \right)^{\frac{1}{2}}
	$$
	since the Rademacher mappings are orthonormal in $L_2([0,1])$. Hence  $\RAD(\R) = \ell_2(\R)$ isometrically. Taking into account that $c_{00}(\R)$ is dense in $\ell_2(\R)$, we also have $\Rad(\R) = \ell_2(\R)$ isometrically (see \cite[p. 234]{DieJarTon-95}). In particular, $\hat{h}((e_n)_{n\in\N}) \notin \RAD(\R)$. Thus $\hat{h}:\RAD(c_0(\R)) \to \RAD(\R)$ is not well-defined, and then $\RAD$ is not Lipschitz stable.

	Let us now prove that $\Rad$ is not Lipschitz stable either. Consider the sequence $(u_n)_{n\in\N}$ given by $u_k = \left(\dfrac{e_n^{(k)}}{\sqrt{k}}\right)_{n\in\N}$ for all $k\in\N$. For each $m\in\N$, let $(u_n^m)_{n\in\N}$ be the sequence defined as
	$$
	u_k^m = \begin{cases}
	u_k \quad & k\leq m,\\
	0 \quad & k > m.
	\end{cases}
	$$

	Clearly $(u_n^m)_{n\in\N} \in c_{00}(c_0(\R))$ and
	\begin{align*}
	\|(u_n)_{n\in\N}-(u_n^m)_{n\in\N}\|_{\RAD(c_0(\R))} &= \sup_{N \geq m+1} \left(\int_0^1 \left\| \sum_{k=m+1}^N \frac{r_k(t)}{\sqrt{k}}(e_n^{(k)})_{n\in\N} \right\|_\infty^2 dt \right)^{\frac{1}{2}}\\
	&= \sup_{N \geq m+1} \left( \int_0^1 \left( \max_{m+1 \leq k \leq N} \frac{1}{\sqrt{k}} \right)^2 dt \right)^{\frac{1}{2}} = \frac{1}{\sqrt{m+1}} \xrightarrow{m\to\infty} 0.
	\end{align*}

	Since $\Rad(c_0(\R))= \overline{c_{00}(c_0(\R))}^{\left\|\cdot\right\|_{\RAD(c_0(\R))}}$ it follows that $(u_n)_{n\in\N} \in \Rad(c_0(\R))$. However 
	$$
	\hat{h}((u_n)_{n\in\N}) = \left( h((e_n^{(1)})_{n\in\N}), h\left(\left(e_n^{(2)}/\sqrt{2}\right)_{n\in\N}\right),\ldots \right) = \left( \frac{1}{\sqrt{n}}\right)_{n\in\N},
	$$
	and this sequence does not belong to $\ell_2(\R)=\Rad(\R)$. Consequently $\hat{h}:\Rad(c_0(\R)) \to \Rad(\R)$ is not well-defined, and we conclude that $\Rad$ is not Lipschitz stable.
\end{proof}

With respect to the sequence class $\ell_p^w$, the exceptional case $p=1$, which still survives in the multilinear setting (see \cite[Theorem 4.3]{BotCam-17}), disappears completely in the Lipschitz context.

\begin{proposition}\label{Prop new}
	For every $p \in [1,\infty)$, the sequence class $\ell_p^w$ is not Lipschitz stable.
\end{proposition}

\begin{proof}
	Let $h: c_0(\R) \to \R$ be defined by
	$$
	h((x_n)_{n\in\N}) = \|(x_n)_{n\in\N}\|_\infty \qquad ((x_n)_{n\in\N}\in c_0(\R)).
	$$

	According to the previous proof, it is straightforward that $h\in\Lip(c_0(\R))$. Let $(e_n)_{n\in\N}$ be the canonical basis of $c_0(\R)$. Since $c_0(\R)^* = \ell_1(\R)$, then for every $\phi = (\phi_n)_{n\in\N} \in B_{\ell_1(\R)}$ we have
	$$
	\left( \sum_{i=1}^\infty |\phi((e^{(i)}_n)_{n\in\N})|^p \right)^{\frac{1}{p}} = \left( \sum_{i=1}^{\infty} |\phi_i|^p \right)^{\frac{1}{p}} \leq \sum_{i=1}^{\infty} |\phi_i| = \|(\phi_n)_{n\in\N}\|_{\ell_1(\R)} \leq 1.
	$$

	Hence $(e_n)_{n\in\N} \in \ell_p^w(c_0(\R))$. However,
	\begin{align*}
	\hat{h}((e_n)_{n\in\N}) &= (h((e^{(1)}_n)_{n\in\N}),h((e^{(2)}_n)_{n\in\N}),\ldots)\\ 
	&= (\|(e_n^{(1)})_{n\in\N}\|_\infty,\|(e_n^{(2)})_{n\in\N}\|_\infty, \ldots)\\
	&= (1,1,\ldots) \notin \ell_p(\R) = \ell_p^w(\R).
	\end{align*}

	Therefore $\hat{h}: \ell_p^w(c_0(\R)) \to \ell_p^w(\R)$ is not well-defined, and so the sequence class $\ell_p^w$ is not Lipschitz stable.
\end{proof}

The theory of operator ideals was initially developed by Pietsch in \cite{Pie-80} and subsequently extended to the Lipschitz setting by Achour et al. in \cite{AchRueSanYah-16}. Let us recall that a Banach Lipschitz ideal $\A_{\Lip}$ is a subclass of the class $\Lip$ endowed with a map $\left\|\cdot\right\|_{\A_{\Lip}}: \A_{\Lip} \to \mathbb{R}_0^+$ so that for every pair of Banach spaces $(X,E)$, the components $\A_{\Lip}(X,E) := \Lip(X,E)\cap \A_{\Lip}$ verify the following properties:
\begin{itemize}
	\item[$(i)$] $(\A_{\Lip}(X,E),\left\|\cdot\right\|_{\A_{\Lip}})$ is a Banach space and $\Lipnorm(f) \leq \|f\|_{\A_{\Lip}}$ for all $f \in \A_{\Lip}(X,E)$.
	\item[$(ii)$] Given $h \in \Lip(X)$ and $y \in E$, the mapping $h \cdot y: X \to E$ defined by $x \mapsto h(x)y$ is in $\A_{\Lip}(X,E)$ and $\|h \cdot y\|_{\A_{\Lip}} = \Lipnorm(h)\|y\|$.
	\item[$(iii)$] The Lipschitz ideal property: Given two Banach spaces $X_0$ and $E_0$, $T \in \L(E,E_0)$, $f \in \A_{\Lip}(X,E)$ and $h \in \Lip(X_0,X)$, then $T\circ f\circ h \in \A_{\Lip}(X_0,E_0)$ with $\|T\circ f\circ h\|_{\A_{\Lip}}\leq \|T\|\|f\|_{\A_{\Lip}}\Lipnorm(h)$.
\end{itemize}

We now prove that the class of $(Z,Y)$-summing Lipschitz functions is a Banach Lipschitz ideal.

\begin{theorem}\label{Th 8}
	Let $Z$ and $Y$ be Lipschitz stable sequence classes such that $Z(\K) \subseteq Y(\K)$ with $\left\|\cdot\right\|_{Y(\K)} \leq \left\|\cdot\right\|_{Z(\K)}$. Then $\left[\Pi_{(Z,Y)}^{\Lip},\pi_{(Z,Y)}^{\Lipnorm}\right]$ is a Banach Lipschitz ideal.
\end{theorem}

\begin{proof}
	$(i):$ Firstly, we show that $\Pi_{(Z,Y)}^{\Lip}(X,E)$ is a linear subspace of $\Lip(X,E)$. Towards this end, let $f,g \in \Pi_{(Z,Y)}^{\Lip}(X,E)$ and $\lambda \in \mathbb{K}$. Then
	$$
	((f+\lambda g)(x_n)-(f+\lambda g)(y_n))_{n\in\N} = (f(x_n)-f(y_n))_{n\in\N} +\lambda (g(x_n)-g(y_n))_{n\in\N} \in Y(E)
	$$
	for all $(x_n)_{n\in\N},(y_n)_{n\in\N} \in Z(X)$. Thus $f+\lambda g \in \Pi_{(Z,Y)}^{\Lip}(X,E)$.

	Let us prove that $\pi_{(Z,Y)}^{\Lipnorm}$ is a norm in the space of $(Z,Y)$-summing Lipschitz functions. It is clear that $\pi_{(Z,Y)}^{\Lipnorm}(f)=0$ whenever $f=0$. On the other hand, assume that $\pi_{(Z,Y)}^{\Lipnorm}(f)=0$. By Lemma \ref{Lem 6} we have that $\Lipnorm(f) \leq \pi_{(Z,Y)}^{\Lipnorm}(f)=0$, i.e., $\Lipnorm(f)=0$ and then $f=0$. In the same vein, let $\alpha \in \mathbb{K}$ and $f\in \Pi_{(Z,Y)}^{\Lip}(X,E)$. Note that
	$$
	\pi_{(Z,Y)}^{\Lipnorm}(\alpha f) = \Lipnorm(\widehat{\alpha f}) = \Lipnorm(\alpha \hat{f}) = |\alpha|\Lipnorm(\hat{f})= |\alpha|\pi_{(Z,Y)}^{\Lipnorm}(f).
	$$

	Finally, given $f,g \in \Pi_{(Z,Y)}^{\Lip}(X,E)$, the triangle inequality is derived as follows:
	$$
	\pi_{(Z,Y)}^{\Lipnorm}(f+g) = \Lipnorm(\widehat{f+g}) = \Lipnorm(\hat{f}+\hat{g}) \leq \Lipnorm(\hat{f})+\Lipnorm(\hat{g}) = \pi_{(Z,Y)}^{\Lipnorm}(f)+\pi_{(Z,Y)}^{\Lipnorm}(g).
	$$

	Hence we have shown that $\left(\Pi_{(Z,Y)}^{\Lip}(X,E),\pi_{(Z,Y)}^{\Lipnorm}\right)$ is a normed space. Now, let $(f_n)_{n\in\N}$ be a Cauchy sequence in $\Pi_{(Z,Y)}^{\Lip}(X,E)$. By Lemma \ref{Lem 6} we have that $(f_n)_{n\in\N}$ is a Cauchy sequence in $\Lip(X,E)$ as well, so there is $f \in \Lip(X,E)$ such that $f_n\xrightarrow{\Lipnorm} f$. Since for each $n\in\N$ the function $f_n \in \Pi_{(Z,Y)}^{\Lip}(X,E)$, then the map $\widehat{f_n}: Z(X) \to Y(E)$ is a well-defined Lipschitz operator such that $\Lipnorm(\widehat{f_n})=\pi_{(Z,Y)}^{\Lipnorm}(f_n)$. Thus $(\widehat{f_n})_{n\in\N}$ is a Cauchy sequence in the Banach space $\Lip(Z(X),Y(E))$ and we can assure the existence of $g \in \Lip(Z(X),Y(E))$ such that $\widehat{f_n} \xrightarrow{\Lipnorm} g$. If we can prove that $\hat{f}=g$ holds, then we obtain that $f$ is $(Z,Y)$-summing Lipschitz, and the completeness of $\pi_{(Z,Y)}^{\Lipnorm}$ easily follows. To this end, let us consider the induced maps $\widetilde{f_n},\tilde{f}:\ell_\infty(X)\to \ell_\infty(E)$. By Theorem \ref{Th 1} it is clear that they are well-defined Lipschitz operators. Since $\ell_\infty$ is a Lipschitz stable sequence class by Proposition \ref{Prop 7}, we deduce that $\Lipnorm(\widetilde{f_n}-\tilde{f})=\Lipnorm(\widetilde{f_n-f})=\Lipnorm(f_n-f) \to 0$, and then $\widetilde{f_n} \xrightarrow{\Lipnorm} \tilde{f}$. Let $(x_n)_{n\in\N} \in Z(X)$. Clearly $\tilde{f}((x_n)_{n\in\N})=\hat{f}((x_n)_{n\in\N})$ and we claim that $\hat{f}((x_n)_{n\in\N})=g((x_n)_{n\in\N})$. Indeed, since $\widetilde{f_n} \xrightarrow{\Lipnorm} \tilde{f}$ then $\widetilde{f_n}((x_n)_{n\in\N}) \xrightarrow{\left\|\cdot\right\|_\infty} \tilde{f}((x_n)_{n\in\N})$. Similarly, from $\widehat{f_n} \xrightarrow{\Lipnorm} g$ we have that $\widehat{f_n}((x_n)_{n\in\N}) \xrightarrow{\left\|\cdot\right\|_{Y(E)}} g((x_n)_{n\in\N})$. Due to $Y(E) \subseteq \ell_\infty(E)$ with $\left\|\cdot\right\|_\infty \leq \left\|\cdot\right\|_{Y(E)}$, it follows that $\widehat{f_n}((x_n)_{n\in\N}) \xrightarrow{\left\|\cdot\right\|_{\infty}} g((x_n)_{n\in\N})$. This gives that $\hat{f}((x_n)_{n\in\N})=g((x_n)_{n\in\N})$ by the uniqueness of the limit of a sequence, as desired. Hence $\hat{f} \in \Lip(Z(X),Y(E))$ and we deduce that $f \in \Pi_{(Z,Y)}^{\Lip}(X,E)$. Furthermore
	$$
	\pi_{(Z,Y)}^{\Lipnorm}(f_n-f) = \Lipnorm(\widehat{f_n-f}) = \Lipnorm(\widehat{f_n}-\hat{f})=\Lipnorm(\widehat{f_n}-g) \to 0,
	$$
	proving that $f_n \xrightarrow{\pi_{(Z,Y)}^{\Lipnorm}} f$.

	$(ii):$ Let $h \in \Lip(X)$ and $y\in E$. We want to prove that $h\cdot y \in \Pi_{(Z,Y)}^{\Lip}(X,E)$. To this end, it suffices to show $\widehat{h \cdot y} \in \Lip(Z(X),Y(E))$. Firstly, the condition established in the hypothesis allows us to guarantee that $\hat{h}((x_n)_{n\in\N}) = (h(x_n))_{n\in\N} \in Z(\K) \subseteq Y(\K)$ for all $(x_n)_{n\in\N} \in Z(X)$. Then
	\begin{align*}
		&\|(\widehat{h \cdot y})((x_n)_{n\in\N})-(\widehat{h \cdot y})((y_n)_{n\in\N})\|_{Y(E)} = \|((h \cdot y)(x_n))_{n\in\N}-((h\cdot y)(y_n))_{n\in\N}\|_{Y(E)}\\
		&= \|(h(x_n)y)_{n\in\N}-(h(y_n)y)_{n\in\N}\|_{Y(E)}= \|y(h(x_n)-h(y_n))_{n\in\N}\|_{Y(E)}\\
		&= \|y\|_E \|(h(x_n)-h(y_n))_{n\in\N}\|_{Y(\mathbb{K})}\leq\|y\|_E \|\hat{h}((x_n)_{n\in\N})-\hat{h}((y_n)_{n\in\N})\|_{Z(\mathbb{K})}\\
		 &\leq \Lipnorm(h)\|y\|_E\|(x_n-y_n)_{n\in\N}\|_{Z(X)}
	\end{align*}
	for all $(x_n)_{n\in\N},(y_n)_{n\in\N} \in Z(X)$, where last inequality is obtained since $h\in\Lip(X) = \Pi_{(Z,Z)}^{\Lip}(X)$ with $\Lipnorm(h)=\Lipnorm(\hat{h})$ by the Lipschitz stability of $Z$. Hence $\widehat{h \cdot y} \in \Lip(Z(X),Y(E))$ with $\Lipnorm(\widehat{h \cdot y}) \leq \Lipnorm(h)\|y\|_E$, and then $h\cdot y \in \Pi_{(Z,Y)}^{\Lip}(X,E)$ with $\pi_{(Z,Y)}^{\Lipnorm}(h\cdot y)\leq \Lipnorm(h)\|y\|_E$. For the remaining inequality, we apply Lemma \ref{Lem 6} to obtain
	$$
	\Lipnorm(h)\|y\|_E=\Lipnorm(h\cdot y) \leq \pi_{(Z,Y)}^{\Lipnorm}(h\cdot y).
	$$

	$(iii):$ Let $T \in \L(E,E_0)$, $f \in \Pi_{(Z,Y)}^{\Lip}(X,E)$ and $h \in \Lip(X_0,X)$. We claim that $\widehat{T \circ f \circ h} \in \Lip(Z(X_0),Y(E_0))$. Indeed, let $(x_n)_{n\in\N},(y_n)_{n\in\N} \in Z(X_0)$ and note that
	\begin{align*}
		&\|(\widehat{T \circ f \circ h})((x_n)_{n\in\N})-(\widehat{T \circ f \circ h})((y_n)_{n\in\N})\|_{Y(E_0)}\\
		&= \|((T\circ f \circ h)(x_n))_{n\in\N}-((T\circ f \circ h)(y_n))_{n\in\N}\|_{Y(E_0)}\\
		&= \|(T(f(h(x_n))-f(h(y_n))))_{n\in\N}\|_{Y(E_0)}\leq \|T\| \|(f(h(x_n))-f(h(y_n)))_{n\in\N}\|_{Y(E)}\\
		&= \|T\| \|\hat{f}((h(x_n))_{n\in\N})-\hat{f}((h(y_n))_{n\in\N})\|_{Y(E)}\leq \|T\|\Lipnorm(\hat{f})\|(h(x_n))_{n\in\N}-(h(y_n))_{n\in\N}\|_{Z(X)}\\
		&= \|T\|\Lipnorm(\hat{f})\|\hat{h}((x_n)_{n\in\N})-\hat{h}((y_n)_{n\in\N})\|_{Z(X)}\leq \|T\|\Lipnorm(\hat{f})\Lipnorm(\hat{h})\|(x_n-y_n)_{n\in\N}\|_{Z(X_0)}\\
		&= \|T\|\pi_{(Z,Y)}^{\Lipnorm}(f)\Lipnorm(h)\|(x_n-y_n)_{n\in\N}\|_{Z(X_0)},
	\end{align*}
	by applying the Lipschitz stability of the sequence classes. Hence $\widehat{T \circ f \circ h}\in \Lip(Z(X_0),Y(E_0))$ with $\Lipnorm(\widehat{T \circ f \circ h})\leq \|T\|\pi_{(Z,Y)}^{\Lipnorm}(f)\Lipnorm(h)$. Thus $T\circ f \circ h \in \Pi_{(Z,Y)}^{\Lip}(X_0,E_0)$ with $\pi_{(Z,Y)}^{\Lipnorm}(T\circ f \circ h) \leq \|T\|\pi_{(Z,Y)}^{\Lipnorm}(f)\Lipnorm(h)$.
\end{proof}

\begin{remark}
Note that condition $Z(\K) \subseteq Y(\K)$ with $\left\|\cdot\right\|_{Y(\K)} \leq \left\|\cdot\right\|_{Z(\K)}$ in Theorem \ref{Th 8} is mandatory. Let us consider $Z=\ell_2$, $Y=\ell_1$, $X=E=\R$, $h=\mathrm{Id}_\R \in \Lip(\R)$ and $y=1$. By Proposition \ref{Prop 7}, $\ell_1$ and $\ell_2$ are Lipschitz stable sequence classes. The induced operator $\widehat{h \cdot y}: \ell_2(\R) \to \ell_1(\R)$ is given by
$$
(\widehat{h \cdot y})((a_n)_{n\in\N}) = ((h \cdot y)(a_n))_{n\in\N} = (a_n)_{n\in\N} \qquad ((a_n)_{n\in\N} \in \ell_2(\R)).
$$

Nevertheless, this operator is not well-defined, and to see this, it is sufficient to consider the sequence $(1/n)_{n\in\N} \in \ell_2(\R)\backslash \ell_1(\R)$. Thus $h \cdot y \notin \Pi_{(\ell_2,\ell_1)}^{\Lip}(\R,\R)$ and axiom $(ii)$ fails.
\end{remark}

We conclude this section with a result that once again highlights the evident differences between the linear/multilinear framework and the Lipschitz setting in which we are working. According to \cite[Example 2.11]{BotCamNas-24}, the Banach ideal $[\Pi_p,\pi_p]$ of absolutely $p$-summing linear operators coincides with the Banach ideal of $(\ell_p^w,\ell_p)$-summing linear operators. Next, we show that, in our context, this relationship is not satisfied in general, and moreover, that $[\Pi_{(\ell_p^w,\ell_p)}^{\Lip},\pi_{(\ell_p^w,\ell_p)}^{\Lipnorm}]$ is not even a Lipschitz ideal.

\begin{proposition}\label{Prop 12}
	Let $p \in [1,\infty)$. Then $\Pi_{(\ell_p^w,\ell_p)}^{\Lip}(X,E) \subseteq \Pi_p^{\Lip}(X,E)$ with $\pi_p^{\Lipnorm}(f)\leq \pi_{(\ell_p^w,\ell_p)}^{\Lipnorm}(f)$ for all $f \in \Pi_{(\ell_p^w,\ell_p)}^{\Lip}(X,E)$. Moreover, the inclusion is strict in general. In particular, $[\Pi_{(\ell_p^w,\ell_p)}^{\Lip},\pi_{(\ell_p^w,\ell_p)}^{\Lipnorm}]$ is not a Banach Lipschitz ideal.
\end{proposition}

\begin{proof}
	Let $f \in \Pi_{(\ell_p^w,\ell_p)}^{\Lip}(X,E)$. Then $\hat{f} \in \Lip(\ell_p^w(X),\ell_p(E))$ with $\Lipnorm(\hat{f}) = \pi_{(\ell_p^w,\ell_p)}^{\Lipnorm}(f)$. Fix $N \in \N$ and let $(x_1,\ldots,x_N)$, $(y_1,\ldots,y_N)\in X^N$. Define the sequences $(u_n)_{n\in\N}$ and $(v_n)_{n\in\N}$ by
	$$
	u_k = \begin{cases}
		x_k \quad &k \leq N,\\
		0 \quad &k > N.
	\end{cases} \qquad
	v_k = \begin{cases}
		y_k \quad &k \leq N,\\
		0 \quad &k > N.
	\end{cases}
	$$

	Clearly $(u_n)_{n\in\N}, (v_n)_{n\in\N} \in \ell_p^w(X)$. Therefore,
	\begin{align*}
		\left(\sum_{k=1}^N \|f(x_k)-f(y_k)\|_E^p \right)^{\frac{1}{p}} &= \|\hat{f}((u_n)_{n\in\N})-\hat{f}((v_n)_{n\in\N})\|_{\ell_p(E)}\\
		&\leq \Lipnorm(\hat{f})\|(u_n-v_n)_{n\in\N}\|_{\ell_p^w(X)}\\
		&= \pi_{(\ell_p^w,\ell_p)}^{\Lipnorm}(f) \sup_{\phi \in B_{X^*}} \left(\sum_{k=1}^N |\phi(x_k-y_k)|^p \right)^{\frac{1}{p}}.
	\end{align*}

	On the other hand, every $\phi \in B_{X^*}$ is a scalar-valued Lipschitz map on $X$ with $\Lipnorm(\phi) = \|\phi\| \leq 1$, and then $B_{X^*} \subseteq B_{\Lip(X)}$. Thus
	$$
	\sup_{\phi \in B_{X^*}}\left( \sum_{k=1}^{N} |\phi(x_k)-\phi(y_k)|^p \right)^{\frac{1}{p}} \leq \sup_{g \in B_{\Lip(X)}} \left( \sum_{k=1}^{N} |g(x_k)-g(y_k)|^p \right)^{\frac{1}{p}}.
	$$

	Consequently $f \in \Pi_p^{\Lip}(X,E)$ with $\pi_p^{\Lipnorm}(f) \leq \pi_{(\ell_p^w,\ell_p)}^{\Lipnorm}(f)$.

	In order to prove that this inclusion is strict in general, let us consider the mapping $h: c_0(\R) \to \R$ given by $h((x_n)_{n\in\N})=\|(x_n)_{n\in\N}\|_\infty$ for all $(x_n)_{n\in\N} \in c_0(\R)$. It is clear that $h \in \Lip(c_0(\R))$ with $\Lipnorm(h) \leq 1$. Therefore, given $N \in \N$ and $(x_1,\ldots,x_N)$, $(y_1,\ldots,y_N) \in c_0(\R)^N$, we have
	$$
	\left( \sum_{k=1}^{N}|h(x_k)-h(y_k)|^p \right)^{\frac{1}{p}} \leq \sup_{g \in B_{\Lip(c_0(\R))}} \left( \sum_{k=1}^{N} |g(x_k)-g(y_k)|^p \right)^{\frac{1}{p}}.
	$$

	Hence $h \in \Pi_p^{\Lip}(c_0(\R),\R)$. Nevertheless, by the proof of Proposition \ref{Prop new}, the canonical basis $(e_n)_{n\in\N} \in \ell_p^w(c_0(\R))$, whereas $\hat{h}((e_n)_{n\in\N}) = (1,1,\ldots) \notin \ell_p(\R)$. So $\hat{h}: \ell_p^w(c_0(\R)) \to \ell_p(\R)$ is not well-defined,  and thus $h \notin \Pi_{(\ell_p^w,\ell_p)}^{\Lip}(c_0(\R),\R)$.

	Finally, assume that $[\Pi_{(\ell_p^w,\ell_p)}^{\Lip},\pi_{(\ell_p^w,\ell_p)}^{\Lipnorm}]$ is a Banach Lipschitz ideal. Then axiom $(ii)$ in the definition of Banach Lipschitz ideal would imply that $g \cdot y \in \Pi_{(\ell_p^w,\ell_p)}^{\Lip}(X,E)$ for all $g \in \Lip(X)$ and all $y \in E$. Just considering $X= c_0(\R)$, $E=\R$, $g=h$ and $y=1$, we obtain that $h \cdot  1 = h \in \Pi_{(\ell_p^w,\ell_p)}^{\Lip}(c_0(\R),\R)$ which is a contradiction.
\end{proof}

\section{Type and cotype of Lipschitz functions}\label{Sec 3}

Let $p,q \in (0,\infty)$ and let $E,F$ be Banach spaces. A bounded linear operator $T\in\L(E,F)$ is said to have:
\begin{itemize}
	\item[$(i)$] Rademacher type $p$ if there exists $C>0$ so that for any $N\in\N$ and $(x_1,\ldots,x_N) \in E^N$,
	\begin{equation}\label{eq 1}
	\left(\int_{0}^{1}\left\| \sum_{k=1}^N r_k(t)T(x_k) \right\|_F^2 dt\right)^{\frac{1}{2}} \leq C \left( \sum_{k=1}^N \|x_k\|_E^p \right)^{\frac{1}{p}}.
	\end{equation}

	The set of all bounded linear operators of Rademacher type $p$ from $E$ into $F$ is denoted as $\T_p(E,F)$, and it becomes a Banach space if we endow it with the norm $\|T\|_{\T_p} = \inf\{C: (\ref{eq 1}) \text{ is satisfied}\}$.
	\item[$(ii)$] Rademacher cotype $q$ if there exists $K>0$ so that for any $N\in\N$ and $(x_1,\ldots,x_N) \in E^N$,
	\begin{equation}\label{eq 2}
	\left(\sum_{k=1}^N \|T(x_k)\|_F^q \right)^{\frac{1}{q}} \leq K \left(\int_{0}^{1}\left\| \sum_{k=1}^{N} r_k(t)x_k \right\|_E^2 dt \right)^{\frac{1}{2}}.
	\end{equation}
	The set of all bounded linear operators of Rademacher cotype $q$ from $E$ into $F$ is denoted by $\C_q(E,F)$, and it becomes a Banach space if we equip this one with the complete norm $\|T\|_{\C_q} = \inf\{K: (\ref{eq 2}) \text{ is satisfied}\}$.
\end{itemize}

Moreover, according to \cite[Chapter 21]{Pie-80} it is well-known that $[\T_p,\left\|\cdot\right\|_{\T_p}]$ and $[\C_q,\left\|\cdot\right\|_{\C_q}]$ are Banach operator ideals.

\subsection{Definition and initial results}

In this section we extend to the Lipschitz setting the notions of type and cotype of linear operators, and we prove that they can be seen as particular cases of $(Z,Y)$-summing Lipschitz functions for certain Lipschitz stable sequence classes $Z$ and $Y$. We also show that, contrary to the multilinear situation, the type classes yield Banach Lipschitz ideals whereas the cotype classes fail to do so.

\begin{definition}\label{Def 2.1}
	Let $0<p,q<\infty$ and $f \in \Lip(X,E)$. Then $f$ is said to have:
	\begin{itemize}
		\item[$(i)$] type $p$ if there exists $C>0$ so that for any $N\in\N$, $(\lambda_1,\ldots, \lambda_N) \in \R^N$ and $(x_1,\ldots, x_N)$, $(y_1,\ldots, y_N) \in X^N$,
		\begin{equation}\label{eq 3}
		\left(\int_{0}^{1}\left\| \sum_{k=1}^{N} \lambda_k r_k(t)(f(x_k)-f(y_k)) \right\|_E^2 dt \right)^{\frac{1}{2}} \leq C \left(\sum_{k=1}^{N} |\lambda_k|^p \|x_k-y_k\|_X^p \right)^{\frac{1}{p}}.
		\end{equation}
		\item[$(ii)$] cotype $q$ if there exists $K>0$ so that for any $N\in\N$, $(\lambda_1,\ldots, \lambda_N) \in \R^N$ and $(x_1,\ldots, x_N)$, $(y_1,\ldots, y_N) \in X^N$,
		\begin{equation}\label{eq 4}
		\left(\sum_{k=1}^{N} |\lambda_k|^q\|f(x_k)-f(y_k)\|_E^q \right)^{\frac{1}{q}} \leq K \left(\int_{0}^{1}\left\| \sum_{k=1}^{N} \lambda_k r_k(t)(x_k-y_k) \right\|_X^2 dt \right)^{\frac{1}{2}}.
		\end{equation}
	\end{itemize}
	The subclasses of zero-preserving Lipschitz functions from $X$ into $E$ having type $p$ and cotype $q$ are denoted by $\T_p^{\Lip}(X,E)$ and $\C_q^{\Lip}(X,E)$, respectively. In addition, we define
		$$
		\|f\|_{\T_p^{\Lipnorm}} = \inf\{C: (\ref{eq 3}) \text{ is satisfied}\}, \qquad \|f\|_{\C_q^{\Lipnorm}} = \inf\{K: (\ref{eq 4}) \text{ is satisfied}\}.
		$$
\end{definition}

In fact, a line of reasoning similar to that in \cite[p. 2989]{FarJoh-09} allows us to ensure that the preceding definitions are the same if we restrict to the case when $\lambda_k=1$ for all $k\in\{1,\ldots, N\}$.

\begin{remark}\label{Rem 2.2}
	\begin{itemize}
		\item[$(i)$] If $p>2$, then we claim that $f=0$ is the only Lipschitz function of type $p$. Indeed, let $f \in \T_p^{\Lip}(X,E)$ and assume that $f \neq 0$. Then given $N\in\N$ and $(x_1,\ldots,x_N)$, $(y_1,\ldots,y_N) \in X^N$ we have
		$$
		\left(\int_{0}^{1} \left\| \sum_{k=1}^{N} r_k(t)(f(x_k)-f(y_k)) \right\|^2_E dt \right)^{\frac{1}{2}} \leq \|f\|_{\T_p^{\Lipnorm}} \left(\sum_{k=1}^{N}\|x_k-y_k\|^p_X \right)^{\frac{1}{p}}.
		$$
		
		Particularly, let us take $x := x_1 = \cdots = x_N \in X$ and $y:= y_1 = \cdots = y_N \in X$ such that $f(x) \neq f(y)$. Then
		\begin{align*}
		&\left(\int_{0}^{1} \left\| \sum_{k=1}^{N} r_k(t)(f(x_k)-f(y_k)) \right\|^2_E dt \right)^{\frac{1}{2}} \leq \|f\|_{\T_p^{\Lipnorm}} \left(\sum_{k=1}^{N}\|x_k-y_k\|^p_X \right)^{\frac{1}{p}}\\
		&\Leftrightarrow \|f(x)-f(y)\|_E \left(\int_{0}^{1} \left|\sum_{k=1}^{N} r_k(t) \right|^2 \right)^{\frac{1}{2}} \leq \|f\|_{\T_p^{\Lipnorm}}\|x-y\|_X N^{\frac{1}{p}}\\
		&\Leftrightarrow \|f(x)-f(y)\|_E N^{\frac{1}{2}} \leq \|f\|_{\T_p^{\Lipnorm}}\|x-y\|_X N^{\frac{1}{p}}\\
		&\Leftrightarrow N^{\frac{1}{2}-\frac{1}{p}} \leq \|f\|_{\T_p^{\Lipnorm}} \frac{\|x-y\|_X}{\|f(x)-f(y)\|_E}.
		\end{align*}

		Since $p>2$ we have that $\frac{1}{2}-\frac{1}{p}>0$, and then $N^{\frac{1}{2}-\frac{1}{p}} \to \infty$ as $N \to \infty$, which is a contradiction.

		\item[$(ii)$] Using the same arguments as in the previous case, we can show that if $q<2$, then the only Lipschitz function of cotype $q$ is $f=0$.

		\item[$(iii)$] If $p \leq 1$, then we claim that any Lipschitz function has type $p$. Indeed, let $f \in \Lip(X,E)$, $N \in \N$ and $(x_1,\ldots,x_N)$, $(y_1,\ldots,y_N)\in X^N$. Then
		\begin{align*}
			\left(\int_{0}^{1} \left\| \sum_{k=1}^N r_k(t)(f(x_k)-f(y_k)) \right\|_E^2 dt \right)^{\frac{1}{2}} &\leq \left(\int_{0}^{1} \left( \sum_{k=1}^{N} |r_k(t)| \|f(x_k)-f(y_k)\|_E \right)^2 dt \right)^{\frac{1}{2}}\\
			&= \sum_{k=1}^{N} \|f(x_k)-f(y_k)\|_E \leq \Lipnorm(f)\sum_{k=1}^{N} \|x_k-y_k\|_X\\
			&\leq \Lipnorm(f)\left(\sum_{k=1}^{N}\|x_k-y_k\|_X^p \right)^{\frac{1}{p}}. 
		\end{align*}

		Hence $f \in \T_p^{\Lip}(X,E)$ with $\|f\|_{\T_p^{\Lipnorm}} = \Lipnorm(f)$.

		\item[$(iv)$] Although Definition \ref{Def 2.1} (ii) is given for the case $q < \infty$, it is not difficult to see that it can be extended in the same way to the case $q = \infty$. In fact, applying the same techniques as in the preceding case, we can prove that if $q=\infty$, then every Lipschitz function has cotype $\infty$.
	\end{itemize}
\end{remark}

In view of the previous comments we will restrict our study on type and cotype of Lipschitz functions to the cases $p \in (1,2]$ and $q \in [2,\infty)$. Moreover, such numbers $p$ and $q$ are called a proper type and a proper cotype for Lipschitz functions, respectively. Next, we introduce some concrete examples of Lipschitz functions having a proper type or cotype.

\begin{example}
	\begin{itemize}
		\item[$(i)$] We claim that, regardless of the Banach space $X$, if $E$ is isomorphic to a Hilbert space $F$, then every function $f\in \Lip(X,E)$ has type $2$ with $\|f\|_{\T_2^{\Lipnorm}} \leq \Lipnorm(f) d(E,F)$, where $d$ denotes the Banach-Mazur distance. Indeed, let $\varepsilon > 0$. Then there exists a linear isomorphism $\varphi: E \to F$ such that $\|\varphi\|\|\varphi^{-1}\| \leq d(E,F)+\varepsilon$. Let $N \in \N$, $(x_1,\ldots,x_N)$, $(y_1,\ldots,y_N) \in X^N$ and note that by the orthogonality of Rademacher functions:
		\begin{align*}
			\left(\int_{0}^{1}\left\| \sum_{k=1}^N r_k(t)(f(x_k)-f(y_k)) \right\|_E^2 dt \right)^{\frac{1}{2}} &\leq \|\varphi^{-1}\|\left(	\int_{0}^{1}\left\|\sum_{k=1}^N r_k(t)\varphi(f(x_k)-f(y_k)) \right\|_F^2 dt \right)^{\frac{1}{2}}\\
			&= \|\varphi^{-1}\|\left( \sum_{k=1}^{N} \|\varphi(f(x_k)-f(y_k))\|_F^2 \right)^{\frac{1}{2}}\\
			&\leq \|\varphi\|\|\varphi^{-1}\| \left(\sum_{k=1}^{N} \|f(x_k)-f(y_k)\|_E^2 \right)^{\frac{1}{2}}\\
			&\leq \Lipnorm(f)(d(E,F)+\varepsilon)\left( \sum_{k=1}^{N} \|x_k-y_k\|_X^2 \right)^{\frac{1}{2}}.
		\end{align*}
	Hence $f \in \T_2^{\Lip}(X,E)$ with $\|f\|_{\T_2^{\Lipnorm}}\leq \Lipnorm(f)(d(E,F)+\varepsilon)$. Just letting $\varepsilon \to 0$ the proof concludes.
		\item[$(ii)$] Similarly, if $X$ is isomorphic to a Hilbert space $Y$, then regardless of the Banach space $E$ we can assure that any $f \in \Lip(X,E)$ has cotype $2$ with $\|f\|_{\C_2^{\Lipnorm}} \leq \Lipnorm(f)d(X,Y)$.
		\item[$(iii)$] Let us consider the map $f = \mathrm{Id}_{\ell_1(\mathbb{R})}$. Assume that there exists $p\in (1,2]$ such that $f\in\T_p^{\Lip}(\ell_1(\R),\ell_1(\R))$. Then given $N\in\N$, $(x_1,\ldots,x_N)$, $(y_1,\ldots,y_N) \in \ell_1(\R)^N$,
		$$
		\left( \int_{0}^{1}\left\| \sum_{k=1}^{N} r_k(t)(f(x_k)-f(y_k)) \right\|_{\ell_1(\R)}^2 dt \right)^{\frac{1}{2}} \leq \|f\|_{\T_p^{\Lipnorm}}\left( \sum_{k=1}^{N} \|x_k-y_k\|_{\ell_1(\R)}^p \right)^{\frac{1}{p}}.
		$$

		Let us take the following particular elements of $\ell_1(\R)^N$
		$$
		(x_1,\ldots, x_N) = ((2e_n^{(1)})_{n\in\N}, \ldots, (2e_n^{(N)})_{n\in\N}), \quad (y_1,\ldots, y_N) = ((e_n^{(1)})_{n\in\N}, \ldots, (e_n^{(N)})_{n\in\N}).
		$$ 
		Then
		\begin{align*}
			&\left( \int_{0}^{1}\left\| \sum_{k=1}^{N} r_k(t)(f(x_k)-f(y_k)) \right\|_{\ell_1(\R)}^2 dt \right)^{\frac{1}{2}} \leq \|f\|_{\T_p^{\Lipnorm}}\left( \sum_{k=1}^{N} \|x_k-y_k\|_{\ell_1(\R)}^p \right)^{\frac{1}{p}}\\
			&\Leftrightarrow \left(\int_{0}^{1} \left\| \sum_{k=1}^{N} r_k(t)(e_n^{(k)})_{n\in\N} \right\|_{\ell_1(\R)}^2 dt \right)^{\frac{1}{2}} \leq \|f\|_{\T_p^{\Lipnorm}}\left(\sum_{k=1}^{N} \|(e_n^{(k)})_{n\in\N}\|_{\ell_1(\R)}^p \right)^{\frac{1}{p}},
		\end{align*}
		and this is equivalent to $N \leq \|f\|_{\T_p^{\Lipnorm}}N^{\frac{1}{p}}$, which is a contradiction. Thus $\mathrm{Id}_{\ell_1(\R)} \in \Lip(\ell_1(\R),\ell_1(\R))$ has no proper type. By similar reasoning it can be shown that $\mathrm{Id}_{c_0(\R)} \in \Lip(c_0(\R),c_0(\R))$ has no proper cotype.
	\end{itemize}
\end{example}

The following inclusion properties between type and cotype of Lipschitz functions can be easily deduced in view of Definition \ref{Def 2.1}.

\begin{proposition}\label{Prop 2.4}
	Let $p,r\in (1,2]$ and $q,s \in [2,\infty)$ be such that $p \leq r$ and $s \leq q$. Then:
	\begin{itemize}
		\item[$(i)$] $\T_r^{\Lip}(X,E) \subseteq \T_p^{\Lip}(X,E)$ with $\|f\|_{\T_p^{\Lipnorm}} \leq \|f\|_{\T_r^{\Lipnorm}}$ for all $f \in \T_r^{\Lip}(X,E)$.
		\item[$(ii)$] $\C_s^{\Lip}(X,E) \subseteq \C_q^{\Lip}(X,E)$ with $\|f\|_{\C_q^{\Lipnorm}} \leq \|f\|_{\C_s^{\Lipnorm}}$ for all $f \in \C_s^{\Lip}(X,E)$. 
	\end{itemize}
\end{proposition}

If we focus on the cases for which the sequence classes coincide with $\Rad, \RAD$ and $\ell_p$ in Theorem \ref{Th 1}, then the following characterizations of the spaces $\T_p^{\Lip}(X,E)$ and $\C_q^{\Lip}(X,E)$ easily follow from Corollary \ref{Cor 2}.

\begin{corollary}\label{Cor 2.6}
	Let $p \in (1,2]$ and $f \in \Lip(X,E)$. The following assertions are equivalent:
	\begin{itemize}
		\item[$(i)$] $f \in \T_p^{\Lip}(X,E)$.
		\item[$(ii)$] $\hat{f} \in\Lip(\ell_p(X),\RAD(E))$.
		\item[$(iii)$] $\hat{f} \in \Lip(\ell_p(X),\Rad(E))$.
	\end{itemize}

	Furthermore, $\Lipnorm(\hat{f}:\ell_p(X) \to \RAD(E)) = \Lipnorm(\hat{f}:\ell_p(X) \to \Rad(E)) = \|f\|_{\T_p^{\Lipnorm}}$.
\end{corollary}

\begin{corollary}\label{Cor 2.7}
	Let $q\in[2,\infty)$ and $f \in \Lip(X,E)$. The following assertions are equivalent:
	\begin{itemize}
		\item[$(i)$] $f \in \C_q^{\Lip}(X,E)$.
		\item[$(ii)$] $\hat{f} \in\Lip(\RAD(X),\ell_q(E))$.
		\item[$(iii)$] $\hat{f} \in \Lip(\Rad(X),\ell_q(E))$.
	\end{itemize}

	Furthermore, $\Lipnorm(\hat{f}:\RAD(X) \to \ell_q(E)) = \Lipnorm(\hat{f}:\Rad(X) \to \ell_q(E)) = \|f\|_{\C_q^{\Lipnorm}}$.
	\end{corollary}

The previous corollaries allow us to ensure that the equalities $\T_p^{\Lip}(X,E) = \Pi_{(\ell_p,\Rad)}^{\Lip}(X,E)$ and $\C_q^{\Lip}(X,E) = \Pi_{(\Rad,\ell_q)}^{\Lip}(X,E)$ hold. Nevertheless, since $\Rad$ is not a Lipschitz stable sequence class, we cannot apply Theorem \ref{Th 8}. Thus, unlike what happens in the multilinear setting, here we are not able to automatically guarantee that these subclasses are Banach Lipschitz ideals, as is the case in \cite[Theorem 3.7]{BotCam-16}. In fact, this leads to another substantial difference: $[\T_p^{\Lip},\left\|\cdot\right\|_{\T_p^{\Lipnorm}}]$ is a Banach Lipschitz ideal, whereas $[\C_q^{\Lip},\left\|\cdot\right\|_{\C_q^{\Lipnorm}}]$ is not.

\begin{theorem}\label{Cor 2.8}
	Let $p\in(1,2]$. Then $[\T_p^{\Lip},\left\|\cdot\right\|_{\T_p^{\Lipnorm}}]$ is a Banach Lipschitz ideal.
\end{theorem}

\begin{proof}
	$(i):$ Firstly, we will show that $\T_p^{\Lip}(X,E)$ is a vector space. Let $f,g \in \T_p^{\Lip}(X,E), \alpha \in \K, N \in \N$ and $(x_1,\ldots,x_N)$, $(y_1,\ldots,y_N) \in X^N$. By Minkowski's inequality
	\begin{align*}
		&\left( \int_0^1\left\| \sum_{k=1}^N r_k(t)((f+\alpha g)(x_k)-(f+\alpha g)(y_k)) \right\|_E^2 dt \right)^{\frac{1}{2}}\\
		&\leq \left( \int_0^1\left\| \sum_{k=1}^N r_k(t)(f(x_k)-f(y_k)) \right\|_E^2 dt \right)^{\frac{1}{2}} + |\alpha| \left( \int_0^1\left\| \sum_{k=1}^N r_k(t)(g(x_k)-g(y_k)) \right\|_E^2 dt \right)^{\frac{1}{2}}\\
		&\leq (\|f\|_{\T_p^{\Lipnorm}}+|\alpha|\|g\|_{\T_p^{\Lipnorm}})\left( \sum_{k=1}^N \|x_k-y_k\|^p_X \right)^{\frac{1}{p}}.
	\end{align*}

	Hence $f+\alpha g \in \T_p^{\Lip}(X,E)$ with $\|f+\alpha g\|_{\T_p^{\Lipnorm}}\leq \|f\|_{\T_p^{\Lipnorm}}+|\alpha|\|g\|_{\T_p^{\Lipnorm}}$. For the homogeneity, it follows immediately from the definition, so $\left\|\cdot\right\|_{\T_p^{\Lipnorm}}$ is a seminorm. To show it is a norm, take $N=1$. Then for every $f \in \T_p^{\Lip}(X,E)$ and every $x,y \in X$, we have $\|f(x)-f(y)\|_E \leq \|f\|_{\T_p^{\Lipnorm}}\|x-y\|_X$. Therefore $\Lipnorm(f) \leq \|f\|_{\T_p^{\Lipnorm}}$. Thus if $\|f\|_{\T_p^{\Lipnorm}}=0$ then $\Lipnorm(f)=0$ and we obtain $f=0$.

	It remains to prove completeness. Towards this end, let $(f_n)_{n\in\N}$ be a Cauchy sequence in $(\T_p^{\Lip}(X,E),\left\|\cdot\right\|_{\T_p^{\Lipnorm}})$. By the above comments we have that $(f_n)_{n\in\N}$ is a Cauchy sequence in the Banach space $(\Lip(X,E),\Lipnorm)$ as well, so there exists $f \in \Lip(X,E)$ such that $f_n \xrightarrow{\Lipnorm} f$. For any $\varepsilon > 0$, there is $n_0 \in \N$ so that $\|f_n-f_m\|_{\T_p^{\Lipnorm}} \leq \varepsilon$ for all $n,m\geq n_0$. Fix $n \geq n_0$ and let $N \in \N$, $(x_1,\ldots,x_N)$ and $(y_1,\ldots,y_N) \in X^N$. Then for every $m \geq n_0$ we have
	$$
	\left(\int_0^1 \left\| \sum_{k=1}^{N} r_k(t)((f_n-f_m)(x_k)-(f_n-f_m)(y_k)) \right\|_E^2 dt \right)^{\frac{1}{2}} \leq \varepsilon \left( \sum_{k=1}^{N} \|x_k-y_k\|_X^p \right)^{\frac{1}{p}}.
	$$

	Now let $m \to \infty$. Since the sum is finite and $f_m \xrightarrow{\Lipnorm} f$, the integrand converges pointwise to the one corresponding to $f_n-f$. Moreover, if $M=\sup_{m\in\N}\Lipnorm(f_m)<\infty$, then for every $m\geq n_0$ and $t\in[0,1]$,
	$$
	\left\|\sum_{k=1}^N r_k(t)((f_n-f_m)(x_k)-(f_n-f_m)(y_k))\right\|_E^2 \leq (M+\Lipnorm(f_n))^2\left(\sum_{k=1}^N\|x_k-y_k\|_X\right)^2. 
	$$

	Thus we can apply Dominated Convergence Theorem to obtain
	\begin{align*}
	&\left(\int_0^1 \left\|\sum_{k=1}^N r_k(t)((f_n-f)(x_k)-(f_n-f)(y_k))\right\|_E^2 dt\right)^{\frac{1}{2}}\\
	&= \left( \lim_{m\to\infty} \int_0^1 \left\|\sum_{k=1}^N r_k(t)((f_n-f_m)(x_k)-(f_n-f_m)(y_k))\right\|_E^2 dt\right)^{\frac{1}{2}}\\
	&\leq \varepsilon \left( \sum_{k=1}^{N} \|x_k-y_k\|_X^p \right)^{\frac{1}{p}}.
	\end{align*}

	Hence $\|f_n-f\|_{\T_p^{\Lipnorm}} \leq \varepsilon$ for all $n \geq n_0$. In particular, $f \in \T_p^{\Lip}(X,E)$ and $f_n \xrightarrow{\left\|\cdot\right\|_{\T_p^{\Lipnorm}}} f$, as desired.

	$(ii):$ Let $h \in \Lip(X)$ and $y \in E$. Given $N \in \N$ and $(x_1,\ldots,x_N)$, $(y_1,\ldots,y_N) \in X^N$, orthogonality of Rademacher mappings gives
	\begin{align*}
		\left( \int_0^1 \left\| \sum_{k=1}^{N} r_k(t)((h\cdot y)(x_k)-(h\cdot y)(y_k)) \right\|_E^2 dt \right)^{\frac{1}{2}} &= \|y\|_E \left( \int_0^1 \left| \sum_{k=1}^{N} r_k(t)(h(x_k)-h(y_k)) \right|^2 dt \right)^{\frac{1}{2}}\\
		&=\|y\|_E \left( \sum_{k=1}^N |h(x_k)-h(y_k)|^2 \right)^{\frac{1}{2}}\\
		&\leq \Lipnorm(h)\|y\|_E\left( \sum_{k=1}^{N} \|x_k-y_k\|^2_X \right)^{\frac{1}{2}}\\
		&\leq \Lipnorm(h)\|y\|_E\left( \sum_{k=1}^{N} \|x_k-y_k\|^p_X \right)^{\frac{1}{p}}.
	\end{align*}

	Consequently $h\cdot y \in \T_p^{\Lip}(X,E)$ with $\|h\cdot y\|_{\T_p^{\Lipnorm}} \leq \Lipnorm(h)\|y\|_E$. For the reverse inequality, note that $\Lipnorm(h)\|y\|_E = \Lipnorm(h\cdot y) \leq \|h \cdot y\|_{\T_p^{\Lipnorm}}$.

	$(iii):$ Let $X_0$ and $E_0$ be Banach spaces, let $T \in \L(E,E_0)$, $f\in\T_p^{\Lip}(X,E)$ and $h \in \Lip(X_0,X)$. Given $N \in \N$ and $(x_1,\ldots,x_N)$, $(y_1,\ldots,y_N) \in X^N$, we have
	\begin{align*}
		&\left( \int_0^1 \left\| \sum_{k=1}^{N} r_k(t)((T\circ f \circ h)(x_k)-(T\circ f \circ h)(y_k)) \right\|_{E_0}^2 dt \right)^{\frac{1}{2}}\\
		&\leq \|T\|\left( \int_0^1 \left\| \sum_{k=1}^{N} r_k(t)(f(h(x_k))-f(h(y_k))) \right\|_E^2 dt \right)^{\frac{1}{2}}\\
		&\leq \|T\|\|f\|_{\T_p^{\Lipnorm}}\left( \sum_{k=1}^{N} \|h(x_k)-h(y_k)\|_X^p \right)^{\frac{1}{p}}\\
		&\leq \|T\|\|f\|_{\T_p^{\Lipnorm}}\Lipnorm(h)\left( \sum_{k=1}^{N} \|x_k-y_k\|_{X_0}^p \right)^{\frac{1}{p}}.
	\end{align*}

	Hence $T\circ f \circ h \in \T_p^{\Lip}(X_0,E_0)$ with $\|T\circ f \circ h\|_{\T_p^{\Lipnorm}} \leq \|T\|\|f\|_{\T_p^{\Lipnorm}}\Lipnorm(h)$.
	\end{proof}

	\begin{proposition}
		Let $q \in [2,\infty)$. Then $[\C_q^{\Lip},\left\|\cdot\right\|_{\C_q^{\Lipnorm}}]$ is not a Banach Lipschitz ideal.
	\end{proposition}

	\begin{proof}
		Let $h:c_0(\R) \to \R$ given by $(x_n)_{n\in\N} \mapsto \|(x_n)_{n\in\N}\|_\infty$. We know that $h \in \Lip(c_0(\R))$. Let $(e_n)_{n\in\N}$ the canonical basis of $c_0(\R)$. For every $N\in\N$, choose $x_k = (e_n^{(k)})_{n\in\N}$ and $y_k = (0)_{n\in\N}$ for $k=1,\ldots, N$. Then
		$$
		\left( \sum_{k=1}^{N} \left|h((e_n^{(k)})_{n\in\N})-h((0)_{n\in\N})\right|^q \right)^{\frac{1}{q}} = N^{\frac{1}{q}},
		$$
		whereas
		$$
		\left( \int_0^1 \left\| \sum_{k=1}^{N} r_k(t)(e_n^{(k)})_{n\in\N} \right\|_\infty^2 dt \right)^{\frac{1}{2}} = 1.
		$$

		Therefore we cannot find any constant $K>0$ for which the cotype inequality is satisfied for this map, and hence $h \notin \C_q^{\Lip}(c_0(\R),\R)$.

		Let us assume that $[\C_q^{\Lip},\left\|\cdot\right\|_{\C_q^{\Lipnorm}}]$ is a Banach Lipschitz ideal. By axiom $(ii)$, for every pair of Banach spaces $(X,E)$, every $g \in \Lip(X)$ and every $y \in E$, we have $g\cdot y \in \C_q^{\Lip}(X,E)$ and $\|g\cdot y\|_{\C_q^{\Lipnorm}} = \Lipnorm(g)\|y\|$. Just taking $X=c_0(\R)$, $E=\R$, $g=h$ and $y=1$, we obtain $h\cdot 1 = h \in \C_q^{\Lip}(c_0(\R),\R)$, a contradiction.
	\end{proof}

\subsection{Lipschitz ideal properties of $[\T_p^{\Lip},\left\|\cdot\right\|_{\T_p^{\Lipnorm}}]$: composition and maximality}

Let $[\A,\left\|\cdot\right\|_\A]$ be a Banach operator ideal. According to \cite{Saa-17}, a function $f \in \Lip(X,E)$ is said to be $\A$-factorizable if there exist a Banach space $F$, a map $h \in \Lip(X,F)$ and an operator $S \in \A(F,E)$ such that $f = S \circ h$. The set of all $\A$-factorizable Lipschitz functions from $X$ into $E$ is denoted by $\A \circ \Lip(X,E)$ and it becomes a Banach space endowed with the norm $\|f\|_{\A\circ\Lip} = \inf\{\|S\|_\A \Lipnorm(h)\}$ where the infimum is taken over all such factorizations of $f$. Even more, by \cite[Theorem 2.1]{Saa-17} we have that $[\A\circ \Lip,\left\|\cdot\right\|_{\A\circ\Lip}]$ is a Banach Lipschitz ideal.

Our next goal will be to analyse the relationship between the Banach Lipschitz ideal $[\T_p^{\Lip},\left\|\cdot\right\|_{\T_p^{\Lipnorm}}]$ and the Lipschitz composition ideal $[\T_p\circ \Lip,\left\|\cdot\right\|_{\T_p\circ \Lip}]$. To this end, let us recall some basic notions introduced in \cite{Kal-04}. Given $x\in X$, the evaluation functional $f \mapsto f(x)$ from $\Lip(X)$ into $\mathbb{K}$ will be denoted by $\delta_x$, and it satisfies $\|\delta_x-\delta_y\|=\|x-y\|$ for all $x,y\in X$. Moreover, the mapping $\delta_X:X \to \mathcal{F}(X)$ given by $x \mapsto \delta_x$ is an isometric embedding, where $\F(X)$ denotes the Lipschitz-free space. In fact, for every $f\in \Lip(X,E)$ there exists a unique $T_f \in \L(\F(X),E)$ so that $f = T_f \circ \delta_X$ and $\|T_f\| = \Lipnorm(f)$.

\begin{lemma}\label{Lem 2.9}
	Let $p\in(1,2]$ and $f \in \Lip(X,E)$. If $T_f \in \T_p(\F(X),E)$, then $f \in \T_p^{\Lip}(X,E)$ with $\|f\|_{\T_p^{\Lipnorm}} \leq \|T_f\|_{\T_p}$.
\end{lemma}

\begin{proof}
	 Assume that $T_f \in\T_p(\F(X),E)$. Then given $N\in\N$ and $(x_1,\ldots,x_N)$, $(y_1,\ldots,y_N) \in X^N$ we have
	\begin{align*}
		\left(\int_{0}^{1}\left\| \sum_{k=1}^N r_k(t)(f(x_k)-f(y_k)) \right\|_E^2 dt \right)^{\frac{1}{2}} &= \left(\int_{0}^{1}\left\| \sum_{k=1}^N r_k(t)(T_f(\delta_X(x_k))-T_f(\delta_X(y_k))) \right\|_E^2 dt \right)^{\frac{1}{2}}\\
		&= \left(\int_{0}^{1}\left\| \sum_{k=1}^{N} r_k(t)T_f(\delta_{x_k}-\delta_{y_k}) \right\|_E^2 dt \right)^{\frac{1}{2}}\\ 
		&\leq \|T_f\|_{\T_p}\left(\sum_{k=1}^{N} \|\delta_{x_k}-\delta_{y_k}\|_{\F(X)}^p \right)^{\frac{1}{p}}= \|T_f\|_{\T_p}\left(\sum_{k=1}^{N} \|x_k-y_k\|_{X}^p \right)^{\frac{1}{p}}.
	\end{align*}

	Thus $f \in \T_p^{\Lip}(X,E)$ with $\|f\|_{\T_p^{\Lipnorm}} \leq \|T_f\|_{\T_p}$, as desired.
\end{proof}

\begin{proposition}\label{Prop 2.10}
	Let $p\in(1,2]$. Then $\T_p \circ \Lip(X,E) \subseteq \T_p^{\Lip}(X,E)$ and $\|f\|_{\T_p^{\Lipnorm}} \leq \|f\|_{\T_p \circ \Lip}$ for all $f \in \T_p\circ\Lip(X,E)$.
\end{proposition}

\begin{proof}
	Let $f \in \T_p\circ\Lip(X,E)$. Then $f=S\circ h$, where $S \in \T_p(F,E)$ and $h\in\Lip(X,F)$. Note that
	$$
	T_f \circ \delta_X = f = S \circ h = S \circ T_h \circ \delta_X.
	$$

	Since $\delta_X$ is an isometric embedding we obtain $T_f=S\circ T_h$, and by the ideal property of $\T_p$ (see, e.g., \cite[21.2.2]{Pie-80}) we have that $T_f \in \T_p(\F(X),E)$. Thus by Lemma \ref{Lem 2.9} we deduce $f \in \T_p^{\Lip}(X,E)$ with
	$$
	\|f\|_{\T_p^{\Lipnorm}} \leq \|T_f\|_{\T_p} = \|S \circ T_h\|_{\T_p} \leq \|S\|_{\T_p}\|T_h\|=\|S\|_{\T_p}\Lipnorm(h).
	$$

	Just taking the infimum over all such factorizations of $f$ we conclude that $\|f\|_{\T_p^{\Lipnorm}} \leq \|f\|_{\T_p\circ\Lip}$.
\end{proof}

We now introduce some composition results concerning type and cotype of Lipschitz functions. To this end, we recover the subclasses of bounded linear operators given in Section \ref{Sec 1}.

\begin{proposition}\label{Prop 2.12}
	\begin{itemize}
		\item[$(i)$] Let $q\in[1,\infty)$. If $S \in \C_q(E,F)$ and $f \in \T_q \circ \Lip(X,E)$, then $S \circ f \in \F_2^{\Lip}(X,F)$ with $\|S\circ f\|_{\F_2^{\Lipnorm}} \leq \|S\|_{\C_q}\|f\|_{\T_q\circ \Lip}$.
		\item[$(ii)$] Let $p \in (1,2]$ and $q\in[2,\infty)$. If $S \in \C_q(E,F)$ and $f \in \Pi_p \circ \Lip(X,E)$, then $S\circ f \in \Pi_2^{\Lip}(X,F)$ with $\pi_2^{\Lipnorm}(S \circ f) \leq \|S\|_{\C_q}\|f\|_{\Pi_p\circ\Lip}$.
		\item[$(iii)$] Let $q \in (1,2]$. If $S \in \I_2(E,F)$ and $f \in \C_{q^*} \circ \Lip(X,E)$, then $S \circ f \in \I_q^{\Lip}(X,F)$ with $\|S \circ f\|_{\I_q^{\Lipnorm}} \leq \|S\|_{\I_2}\|f\|_{\C_{q^*}\circ\Lip}$.
	\end{itemize}
\end{proposition}

\begin{proof}
	$(i):$ Since $f \in \T_q\circ\Lip(X,E)$ there exist a map $h\in\Lip(X,Z)$ and $T\in \T_q(Z,E)$ so that $f = T \circ h$. Thus note that $S\circ T \in \C_q \circ \T_q(Z,F)$, and by \cite[21.4.6]{Pie-80} we have that $S \circ T \in \F_2(Z,F)$ with $\|S \circ T\|_{\F_2} \leq \|S\|_{\C_q}\|T\|_{\T_q}$. Consequently $S \circ f = S \circ T \circ h \in \F_2 \circ \Lip(X,F) = \F_2^{\Lip}(X,F)$ by \cite[Proposition 2.9]{AchTia-24}. Furthermore, applying \cite[p. 828]{Saa-17} we obtain
	$$
	\|S\circ f\|_{\F_2^{\Lipnorm}} = \|S \circ f\|_{\F_2 \circ \Lip} \leq \|S \circ T\|_{\F_2}\Lipnorm(h) \leq \|S\|_{\C_q}\|T\|_{\T_q}\Lipnorm(h),
	$$
	and taking the infimum over all the factorizations of $f$ as in the beginning of the proof we conclude that $\|S \circ f\|_{\F_2^{\Lipnorm}} \leq \|S\|_{\C_q}\|f\|_{\T_q\circ \Lip}$.
	
	$(ii):$ This statement can be proven similarly to the first one but replacing \cite[21.4.6]{Pie-80} and \cite[p. 828]{Saa-17} by \cite[21.4.8]{Pie-80} and \cite[Proposition 4.5 $(i)$ and Proposition 4.8]{BelChe-18}, respectively.

	$(iii):$ This assertion follows analogously to $(i)$ but replacing \cite[21.4.6]{Pie-80} and \cite[p. 828]{Saa-17} by \cite[21.4.7]{Pie-80} and \cite[Remark 2.7]{AchRueYah-17}, respectively.
\end{proof}

Finally, we focus on the study of the maximality for the ideal $[\T_p^{\Lip},\left\|\cdot\right\|_{\T_p^{\Lipnorm}}]$. Towards this end, let us recall that a Banach Lipschitz ideal $[\A^{\Lip},\left\|\cdot\right\|_{\A^{\Lip}}]$ is called:
\begin{itemize}
	\item[$(i)$] regular if for any $f \in \Lip(X,E)$ such that $\kappa_E \circ f \in \A^{\Lip}(X,E^{**})$, then $f \in \A^{\Lip}(X,E)$ and $\|f\|_{\A^{\Lip}} = \|\kappa_E \circ f\|_{\A^{\Lip}}$.
	\item[$(ii)$] ultrastable if for every pair $((X_i)_{i\in I},(E_i)_{i\in I})$ of families of Banach spaces, every family of Lipschitz functions $(f_i)_{i\in I}$ so that $f_i \in \A^{\Lip}(X_i,E_i)$ with $\sup_{i\in I} \|f_i\|_{\A^{\Lip}} < \infty$ and every ultrafilter $\mathcal{U}$ of $I$, then the Lipschitz map $(f_i)^{\mathcal{U}}: (X_i)_{\mathcal{U}} \to (E_i)_{\mathcal{U}}$ defined by $(x_i)_{\mathcal{U}} \mapsto (f_i(x_i))_{\mathcal{U}}$ is in $\A^{\Lip}((X_i)_{\mathcal{U}}, (E_i)_{\mathcal{U}})$ with $\|(f_i)^{\mathcal{U}}\|_{\A^{\Lip}} \leq \sup_{i\in I}\|f_i\|_{\A^{\Lip}}$.
	\item[$(iii)$] maximal if $[\A^{\Lip},\left\|\cdot\right\|_{\A^{\Lip}}] = [(\A^{\Lip})^{\mathrm{max}},\left\|\cdot\right\|_{(\A^{\Lip})^{\mathrm{max}}}]$, where the components of this Banach Lipschitz ideal are defined, for each pair of Banach spaces $(X,E)$, by
	\begin{align*}
	(\A^{\Lip})^{\mathrm{max}}(X,E) &= \left\{ f \in \Lip(X,E): \left\|f\right\|_{(\A^{\Lip})^{\mathrm{max}}}=\sup_{\substack{Z \in \mathrm{FIN}(X)\\ Y \in \mathrm{COFIN}(E)}} \|q_Y^E \circ f \circ \iota_Z^X\|_{\A^{\Lip}} <\infty \right\}.
	\end{align*}
\end{itemize}

According to \cite[Theorem 3.6]{AlbTur-24} a Banach Lipschitz ideal is maximal if and only if it is regular and ultrastable. Bearing in mind this result we are in a position to study the maximality of the desired ideal.

\begin{theorem}
	Let $p\in(1,2]$. Then $[\T_p^{\Lip},\left\|\cdot\right\|_{\T_p^{\Lipnorm}}]$ is maximal.
\end{theorem}

\begin{proof}
	Firstly, we show the regularity of $[\T_p^{\Lip},\left\|\cdot\right\|_{\T_p^{\Lipnorm}}]$. Let us suppose that $\kappa_E \circ f \in \T_p^{\Lip}(X,E^{**})$. Since $\kappa_E$ is an isometric embedding we have
	\begin{align*}
		\left(\int_{0}^{1} \left\| \sum_{k=1}^N r_k(t)(\kappa_E(f(x_k))-\kappa_E(f(y_k))) \right\|_{E^{**}}^2 dt \right)^{\frac{1}{2}} &= \left(\int_{0}^{1} \left\| \kappa_E\left( \sum_{k=1}^{N} r_k(t)(f(x_k)-f(y_k)) \right) \right\|_{E^{**}}^2 dt \right)^{\frac{1}{2}}\\
		&= \left(\int_{0}^{1}\left\| \sum_{k=1}^{N} r_k(t)(f(x_k)-f(y_k)) \right\|_E^2 dt \right)^{\frac{1}{2}}\\
		&\leq \|\kappa_E\circ f\|_{\T_p^{\Lipnorm}}\left( \sum_{k=1}^{N} \|x_k-y_k\|_X^p\right)^{\frac{1}{p}}
	\end{align*}
	for any $N\in\N$ and any $(x_1,\ldots,x_N)$, $(y_1,\ldots,y_N) \in X^N$. Thus $f \in \T_p^{\Lip}(X,E)$ with $\|f\|_{\T_p^{\Lipnorm}} \leq \|\kappa_E \circ f\|_{\T_p^{\Lipnorm}}$. The remaining inequality can be obtained from the ideal property of $[\T_p^{\Lip},\left\|\cdot\right\|_{\T_p^{\Lipnorm}}]$.

	Finally, we prove the ultrastability of the Lipschitz ideal $[\T_p^{\Lip},\left\|\cdot\right\|_{\T_p^{\Lipnorm}}]$. Let $(X_i)_{i\in I}$, $(E_i)_{i\in I}$ be two families of Banach spaces, let $(f_i)_{i\in I}$ be a family of Lipschitz functions so that $f_i \in \T_p^{\Lip}(X_i,E_i)$ with $\sup_{i\in I} \|f_i\|_{\T_p^{\Lipnorm}} < \infty$ and let $\mathcal{U}$ be an ultrafilter of $I$. Then, given $N\in\N$ and $((x_1^i)_\U,\ldots, (x_N^i)_\U)$, $((y_1^i)_\U, \ldots, (y_N^i)_\U) \in ((X_i)_\U)^N$, we have
	\begin{align*}
		&\left( \int_{0}^{1} \left\| \sum_{k=1}^{N} r_k(t)\left((f_i)^\U((x_k^i)_\U)-(f_i)^\U((y_k^i)_\U)\right) \right\|_{(E_i)_\U}^2 dt \right)^{\frac{1}{2}}\\
		&= \left( \int_{0}^{1} \left\| \sum_{k=1}^{N} r_k(t)\left((f_i(x_k^i))_\U-(f_i(y_k^i))_\U\right) \right\|_{(E_i)_\U}^2 dt \right)^{\frac{1}{2}}\\
		&= \left( \int_{0}^{1} \left\| \left(\sum_{k=1}^{N} r_k(t)(f_i(x_k^i)-f_i(y_k^i))\right)_\U \right\|_{(E_i)_\U}^2 dt \right)^{\frac{1}{2}}\\
		&= \lim_\U \left( \int_{0}^{1} \left\| \sum_{k=1}^{N} r_k(t)(f_i(x_k^i)-f_i(y_k^i)) \right\|_{E_i}^2 dt \right)^{\frac{1}{2}}\\
		&\leq \lim_\U \left(\|f_i\|_{\T_p^{\Lipnorm}} \left( \sum_{k=1}^{N} \|x_k^i-y_k^i\|_{X_i}^p \right)^{\frac{1}{p}}\right)\\
		&\leq \sup_{i\in I} \|f_i\|_{\T_p^{\Lipnorm}} \lim_\U \left( \sum_{k=1}^{N} \|x_k^i-y_k^i\|_{X_i}^p \right)^{\frac{1}{p}}\\
		&= \sup_{i\in I} \|f_i\|_{\T_p^{\Lipnorm}} \left( \sum_{k=1}^{N} \|(x_k^i)_\U-(y_k^i)_\U\|_{(X_i)_\U}^p \right)^{\frac{1}{p}},
	\end{align*}
	where we have taken into account that $\int_{0}^{1}\left\|\sum_{k=1}^{N}r_k(t)(f_i(x_k^i)-f_i(y_k^i))\right\|_{E_i}^2 dt$ is a finite sum in order to obtain the third equality, and that $\sum_{k=1}^{N}\|x_k^i-y_k^i\|_{X_i}^p$ is finite in order to obtain the last one. Consequently, we can ensure that $(f_i)^\U \in \T_p^{\Lip}((X_i)_\U,(E_i)_\U)$ with $\|(f_i)^\U\|_{\T_p^{\Lipnorm}} \leq \sup_{i\in I} \|f_i\|_{\T_p^{\Lipnorm}}$. Hence $[\T_p^{\Lip},\left\|\cdot\right\|_{\T_p^{\Lipnorm}}]$ is ultrastable. The proof concludes simply by considering \cite[Theorem 3.6]{AlbTur-24}.
	\end{proof}

	\subsection*{Acknowledgements} This work has been developed as a result of a research stay at the University of Lille, and I warmly thank Professor Mostafa Mbekhta for his hospitality throughout this process.

	\subsection*{Funding Declaration} Research partially supported by FPU23/03235 predoctoral fellowship of the Spanish Ministry of Universities and by Junta de Andaluc\'ia grant FQM194.

	\subsection*{Data availability} Not applicable.

	\subsection*{Conflict of interest} There are no relevant financial or non-financial interests to disclose.

\end{document}